\documentclass[11pt,a4paper]{article}

\usepackage{epsf,epsfig,amsfonts,amsgen,amsmath,amstext,amsbsy,amsopn,amsthm
}
\usepackage{amsmath}
\usepackage{amsfonts,amsthm,amssymb}
\usepackage{amsfonts}
\usepackage{graphics}
\usepackage{latexsym,bm}
\usepackage{amsfonts,amsthm,amssymb,bbding}
\usepackage{indentfirst}
\usepackage{graphicx}
\usepackage{color}
\usepackage[colorlinks=true,anchorcolor=blue,filecolor=blue,linkcolor=blue,urlcolor=blue,citecolor=blue]{hyperref}
\usepackage{float}
\usepackage{tikz}
\usepackage[mathscr]{eucal}
\newtheorem{thm}{Theorem}

\newtheorem{lemma}{Lemma}
\newtheorem{false statement}{False statement}

\theoremstyle{definition}

\newtheorem{claim}{Claim}

\newtheorem{case}{Case}
\newtheorem{subcase}{Subcase}[case]

\begin{document}

\title{Spectral Tur\'{a}n problem for $\mathcal{{K}}_{3,3}^{-}$-free signed graphs
\footnote{Supported by Natural Science Foundation of Xinjiang Uygur Autonomous Region (No. 2024D01C41), NSFC (No. 12361071), The Basic scientific research in universities of Xinjiang Uygur Autonomous Region (XJEDU2025P001) and Tianshan Talent Training Program (No. 2024TSYCCX0013).}}
\author{ Mingsong Qin, {Dan Li}\thanks{Corresponding author. E-mail: ldxjedu@163.com.}\\
{\footnotesize  College of Mathematics and System Science, Xinjiang University, Urumqi 830046, China}}
\date{}

\maketitle {\flushleft\large\bf Abstract:}
The classical spectral Tur\'{a}n problem is to determine the maximum spectral radius of an  $\mathcal{F}$-free graph of order $n$. Zhai and Wang [Linear Algebra Appl, 437 (2012) 1641-1647] determined the maximum spectral radius  of ${C}_{4}$-free graphs of given order. Additionally, Nikiforov obtained spectral strengthenings of the K\H{o}vari-S\'{o}s-Tur\'{a}n theorem [Linear Algebra Appl, 432 (2010) 1405-1411] when the forbidden graphs are complete bipartite. The spectral Tur\'{a}n problem concerning forbidden complete bipartite graphs in signed graphs has also attracted considerable attention. Let $\mathcal{K}_{s,t}^-$ be the set of all unbalanced signed graphs with underlying graphs $K_{s,t}$. Since the cases where $s=1$ or $t=1$ do not conform to the definition of  $\mathcal{K}_{s,t}^-$, it follows that $s,t\geq 2$.  Wang and Lin [Discrete Appl. Math, 372 (2025) 164-172] have solved the case of $s=t=2$ since  $\mathcal{K}_{2,2}^-$ is $\mathcal{C}_{4}^{-}$ in this situation. This paper gives an answer for $s=t=3$ and  completely characterizes the corresponding extremal signed graphs.

\vspace{0.1cm}
\begin{flushleft}
\end{flushleft}

\section{Introduction}
 All graphs in this paper are simple. Let $G$ be a graph with vertex set $V(G)=\{v_1,...,v_n\}$ and edge set $E(G)=\{e_1,...,e_m\}$. The order and size of $G$ are defined as $|V(G)|$ and $|E(G)|$, respectively. An underlying graph $G$ together with a sign function $\sigma:E(G)\to\{-1,+1\}$  forms a signed graph $\Gamma=(G,\sigma)$. In a signed graph, edge signs are usually interpreted as $\pm1$. An edge $e$ is positive (resp. negative) if $\sigma(e)=+1$ (resp. $\sigma(e)=-1$). A cycle in $\Gamma$ is said to be positive if it contains an even number of negative edges, otherwise it is negative. $\Gamma=(G,\sigma)$ is balanced if there are no negative cycles, otherwise it is unbalanced. Let $U\subset V(G)$. The operation that changes the signs of all edges between $U$ and $V(G)\setminus U$ is called a switching operation. If a signed graph $\Gamma^\prime$ is obtained from $\Gamma$ by applying finitely many switching operations, then $\Gamma$ is said to be switching equivalent to $\Gamma^\prime$. For more details about the notion of signed graphs, we refer to \cite{1}. Signed graph was first introduced in works of Harary \cite{8} and Cartwright and Harary \cite{5}, and the matroids of graphs were extended to matroids of signed graphs by Zaslavsky \cite{19}. Chaiken \cite{6} and Zaslavsky \cite{19} obtained the Matrix-Tree Theorem for signed graph independently. The theory of signed graphs is a special case of that of gain graphs and of biased graphs \cite{20}. The adjacency matrix of $\Gamma$ is defined as $A(\Gamma)=(a_{ij}^\sigma)$, where $a^\sigma_{ij}=\sigma(v_iv_j)$ if $v_i\sim v_j $, otherwise, $a^\sigma_{ij}=0$. The eigenvalues of $\Gamma$ are written as\begin{figure}[H]
 	\centering
 	\ifpdf
 	\setlength{\unitlength}{1bp}%
 	\begin{picture}(380.48, 288.68)(0,0)
 		\put(0,0){\includegraphics{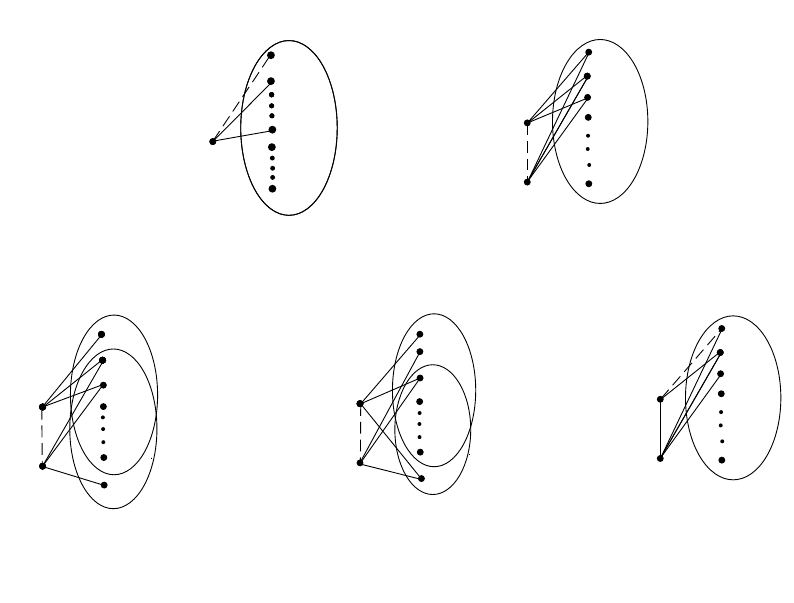}}
 		\put(132.18,246.90){\fontsize{11.38}{13.66}\selectfont $v_2$}
 		\put(133.21,225.29){\fontsize{11.38}{13.66}\selectfont $v_{t-1}$}
 		\put(132.86,215.00){\fontsize{11.38}{13.66}\selectfont $v_t$}
 		\put(133.21,195.28){\fontsize{11.38}{13.66}\selectfont $v_{n-1}$}
 		\put(132.35,259.76){\fontsize{11.38}{13.66}\selectfont $v_1$}
 		\put(92.20,218.75){\fontsize{11.38}{13.66}\selectfont $u$}
 		\put(128.55,175.19){\fontsize{11.38}{13.66}\selectfont $\Gamma_{n,t}$}
 		\put(5.67,60.82){\fontsize{11.38}{13.66}\selectfont $v_2$}
 		\put(52.47,112.73){\fontsize{11.38}{13.66}\selectfont $v_3$}
 		\put(51.95,67.20){\fontsize{11.38}{13.66}\selectfont $v_n$}
 		\put(104.64,239.71){\fontsize{11.38}{13.66}\selectfont $-$}
 		\put(10.08,75.97){\fontsize{11.38}{13.66}\selectfont $-$}
 		\put(125.24,273.62){\fontsize{11.38}{13.66}\selectfont $K_{n-1}$}
 		\put(43.81,142.84){\fontsize{11.38}{13.66}\selectfont $K_{n-3}$}
 		\put(47.66,31.19){\fontsize{11.38}{13.66}\selectfont $Z_1$}
 		\put(240.14,228.18){\fontsize{11.38}{13.66}\selectfont $v_1$}
 		\put(285.10,259.29){\fontsize{11.38}{13.66}\selectfont $v_3$}
 		\put(285.12,249.10){\fontsize{11.38}{13.66}\selectfont $v_4$}
 		\put(285.43,197.72){\fontsize{11.38}{13.66}\selectfont $v_n$}
 		\put(285.17,239.13){\fontsize{11.38}{13.66}\selectfont $v_5$}
 		\put(243.13,215.14){\fontsize{11.38}{13.66}\selectfont $-$}
 		\put(275.90,274.12){\fontsize{11.38}{13.66}\selectfont $K_{n-2}$}
 		\put(282.40,175.05){\fontsize{11.38}{13.66}\selectfont $U$}
 		\put(240.31,198.37){\fontsize{11.38}{13.66}\selectfont $v_2$}
 		\put(118.09,8.12){\fontsize{11.38}{13.66}\selectfont Fig.1.\label{Fig1} The signed graphs $\Gamma_{n,t},U, Z_1, Z_2, W$.}
 		\put(285.28,229.31){\fontsize{11.38}{13.66}\selectfont $v_6$}
 		\put(6.63,88.67){\fontsize{11.38}{13.66}\selectfont $v_1$}
 		\put(52.28,124.39){\fontsize{11.38}{13.66}\selectfont $v_5$}
 		\put(52.73,100.43){\fontsize{11.38}{13.66}\selectfont $v_4$}
 		\put(53.39,90.63){\fontsize{11.38}{13.66}\selectfont $v_7$}
 		\put(52.35,52.40){\fontsize{11.38}{13.66}\selectfont $v_6$}
 		\put(74.48,59.47){\fontsize{11.38}{13.66}\selectfont $K_{n-3}$}
 		\put(158.04,62.42){\fontsize{11.38}{13.66}\selectfont $v_2$}
 		\put(204.66,116.94){\fontsize{11.38}{13.66}\selectfont $v_6$}
 		\put(203.95,70.14){\fontsize{11.38}{13.66}\selectfont $v_n$}
 		\put(162.44,77.57){\fontsize{11.38}{13.66}\selectfont $-$}
 		\put(195.54,143.16){\fontsize{11.38}{13.66}\selectfont $K_{n-3}$}
 		\put(202.22,32.20){\fontsize{11.38}{13.66}\selectfont $Z_2$}
 		\put(159.00,90.27){\fontsize{11.38}{13.66}\selectfont $v_1$}
 		\put(204.65,125.99){\fontsize{11.38}{13.66}\selectfont $v_5$}
 		\put(204.76,104.49){\fontsize{11.38}{13.66}\selectfont $v_3$}
 		\put(204.12,57.19){\fontsize{11.38}{13.66}\selectfont $v_4$}
 		\put(224.95,61.82){\fontsize{11.38}{13.66}\selectfont $K_{n-3}$}
 		\put(204.83,93.27){\fontsize{11.38}{13.66}\selectfont $v_7$}
 		\put(346.88,34.85){\fontsize{11.38}{13.66}\selectfont $W$}
 		\put(301.98,95.52){\fontsize{11.38}{13.66}\selectfont $v_2$}
 		\put(348.94,126.63){\fontsize{11.38}{13.66}\selectfont $v_1$}
 		\put(348.96,116.44){\fontsize{11.38}{13.66}\selectfont $v_3$}
 		\put(349.27,65.06){\fontsize{11.38}{13.66}\selectfont $v_n$}
 		\put(349.01,106.47){\fontsize{11.38}{13.66}\selectfont $v_5$}
 		\put(320.98,112.49){\fontsize{11.38}{13.66}\selectfont $-$}
 		\put(339.74,141.46){\fontsize{11.38}{13.66}\selectfont $K_{n-2}$}
 		\put(303.16,65.71){\fontsize{11.38}{13.66}\selectfont $v_4$}
 		\put(349.12,96.65){\fontsize{11.38}{13.66}\selectfont $v_6$}
 	\end{picture}%
 	\else
 	\setlength{\unitlength}{1bp}%
 	\begin{picture}(380.48, 288.68)(0,0)
 		\put(0,0){\includegraphics{K3,3.pdf}}
 		\put(132.18,246.90){\fontsize{11.38}{13.66}\selectfont $v_2$}
 		\put(133.21,225.29){\fontsize{11.38}{13.66}\selectfont $v_{t-1}$}
 		\put(132.86,215.00){\fontsize{11.38}{13.66}\selectfont $v_t$}
 		\put(133.21,195.28){\fontsize{11.38}{13.66}\selectfont $v_{n-1}$}
 		\put(132.35,259.76){\fontsize{11.38}{13.66}\selectfont $v_1$}
 		\put(92.20,218.75){\fontsize{11.38}{13.66}\selectfont $u$}
 		\put(128.55,175.19){\fontsize{11.38}{13.66}\selectfont $\Gamma_{n,t}$}
 		\put(5.67,60.82){\fontsize{11.38}{13.66}\selectfont $v_2$}
 		\put(52.47,112.73){\fontsize{11.38}{13.66}\selectfont $v_3$}
 		\put(51.95,67.20){\fontsize{11.38}{13.66}\selectfont $v_n$}
 		\put(104.64,239.71){\fontsize{11.38}{13.66}\selectfont $-$}
 		\put(10.08,75.97){\fontsize{11.38}{13.66}\selectfont $-$}
 		\put(125.24,273.62){\fontsize{11.38}{13.66}\selectfont $K_{n-1}$}
 		\put(43.81,142.84){\fontsize{11.38}{13.66}\selectfont $K_{n-3}$}
 		\put(47.66,31.19){\fontsize{11.38}{13.66}\selectfont $Z_1$}
 		\put(240.14,228.18){\fontsize{11.38}{13.66}\selectfont $v_1$}
 		\put(285.10,259.29){\fontsize{11.38}{13.66}\selectfont $v_3$}
 		\put(285.12,249.10){\fontsize{11.38}{13.66}\selectfont $v_4$}
 		\put(285.43,197.72){\fontsize{11.38}{13.66}\selectfont $v_n$}
 		\put(285.17,239.13){\fontsize{11.38}{13.66}\selectfont $v_5$}
 		\put(243.13,215.14){\fontsize{11.38}{13.66}\selectfont $-$}
 		\put(275.90,274.12){\fontsize{11.38}{13.66}\selectfont $K_{n-2}$}
 		\put(282.40,175.05){\fontsize{11.38}{13.66}\selectfont $U$}
 		\put(240.31,198.37){\fontsize{11.38}{13.66}\selectfont $v_2$}
 		\put(118.09,8.12){\fontsize{11.38}{13.66}\selectfont Fig.1. \label{Fig1} The signed graphs $\Gamma_{n,t},U, Z_1, Z_2, W$.}
 		\put(285.28,229.31){\fontsize{11.38}{13.66}\selectfont $v_6$}
 		\put(6.63,88.67){\fontsize{11.38}{13.66}\selectfont $v_1$}
 		\put(52.28,124.39){\fontsize{11.38}{13.66}\selectfont $v_5$}
 		\put(52.73,100.43){\fontsize{11.38}{13.66}\selectfont $v_4$}
 		\put(53.39,90.63){\fontsize{11.38}{13.66}\selectfont $v_7$}
 		\put(52.35,52.40){\fontsize{11.38}{13.66}\selectfont $v_6$}
 		\put(74.48,59.47){\fontsize{11.38}{13.66}\selectfont $K_{n-3}$}
 		\put(158.04,62.42){\fontsize{11.38}{13.66}\selectfont $v_2$}
 		\put(204.66,116.94){\fontsize{11.38}{13.66}\selectfont $v_6$}
 		\put(203.95,70.14){\fontsize{11.38}{13.66}\selectfont $v_n$}
 		\put(162.44,77.57){\fontsize{11.38}{13.66}\selectfont $-$}
 		\put(195.54,143.16){\fontsize{11.38}{13.66}\selectfont $K_{n-3}$}
 		\put(202.22,32.20){\fontsize{11.38}{13.66}\selectfont $Z_2$}
 		\put(159.00,90.27){\fontsize{11.38}{13.66}\selectfont $v_1$}
 		\put(204.65,125.99){\fontsize{11.38}{13.66}\selectfont $v_5$}
 		\put(204.76,104.49){\fontsize{11.38}{13.66}\selectfont $v_3$}
 		\put(204.12,57.19){\fontsize{11.38}{13.66}\selectfont $v_4$}
 		\put(224.95,61.82){\fontsize{11.38}{13.66}\selectfont $K_{n-3}$}
 		\put(204.83,93.27){\fontsize{11.38}{13.66}\selectfont $v_7$}
 		\put(346.88,34.85){\fontsize{11.38}{13.66}\selectfont $W$}
 		\put(301.98,95.52){\fontsize{11.38}{13.66}\selectfont $v_2$}
 		\put(348.94,126.63){\fontsize{11.38}{13.66}\selectfont $v_1$}
 		\put(348.96,116.44){\fontsize{11.38}{13.66}\selectfont $v_3$}
 		\put(349.27,65.06){\fontsize{11.38}{13.66}\selectfont $v_n$}
 		\put(349.01,106.47){\fontsize{11.38}{13.66}\selectfont $v_5$}
 		\put(320.98,112.49){\fontsize{11.38}{13.66}\selectfont $-$}
 		\put(339.74,141.46){\fontsize{11.38}{13.66}\selectfont $K_{n-2}$}
 		\put(303.16,65.71){\fontsize{11.38}{13.66}\selectfont $v_4$}
 		\put(349.12,96.65){\fontsize{11.38}{13.66}\selectfont $v_6$}
 	\end{picture}%
 	\fi
 \end{figure}
\noindent $\lambda_1(A(\Gamma))\geq\lambda_{2}(A(\Gamma))\geq\cdots\geq\lambda_n(A(\Gamma))$  in decreasing  order which are the eigenvalues of $A(\Gamma)$ and $\lambda_1(A(\Gamma))$ is the index of $\Gamma$. 

Given a set $\mathcal{F}$ of graph $G$, if graph $G$ contains no  subgraph isomorphic to any one in $\mathcal{F}$, then $G$ is called $\mathcal{F}$-free. The classical spectral Tur\'{a}n problem is to determine the maximum spectral radius of an $\mathcal{F}$-free graph of order $n$, which is known as the spectral Tur\'{a}n number of $\mathcal{F}$. This problem was originally proposed by Nikiforov \cite{13}. Tur\'{a}n \cite{14} raised and solved the extremal problem for ${K}_{r}$-free graphs with $r\geq 3$. For more on the spectral Tur\'{a}n problem for unsigned graphs see \cite{t2,ZZ2,t6,ZZ3}.

In this paper, we  focus on the spectral Tur\'{a}n problem in signed graphs. It is worth noting that  Brunetti and Stani\'{c} \cite{4} studied the extremal spectral radius among all unbalanced connected signed graphs. For the maximum index of a signed graph, see \cite{QQQ1,3,9,QQ1,12}. Let $\mathcal{K}_r^-$ and $\mathcal{C}_r^-$ be the sets of all unbalanced signed graphs with underlying graphs $K_r$ and $C_r$, respectively.  Chen and Yuan \cite{7} and Wang \cite{k1}  gave the spectral Tur\'{a}n number of $\mathcal{K}_{4}^{-}$ and $\mathcal{K}_{5}^{-}$, respectively. Xiong and Hou \cite{k2} determined the spectral Tur\'{a}n number of $\mathcal{K}_r^-$ for $6\leq r<\frac n2$. In 2022, Wang, Hou and Li \cite{15} determined the spectral Tur\'{a}n number of $\mathcal{C}_{3}^{-}$.  Moreover, the $\mathcal{C}_{2k+1}^{-}$-free unbalanced signed graphs of fixed order $n$ with maximum index have been determined in \cite{18}, where $3 \leq k \leq n/10-1$. Let $\mathcal{K}_{s,t}^-$ be the set of all unbalanced signed graphs with underlying graphs $K_{s,t}$. Motivated by these works, we focus on the spectral Tur\'{a}n problem of $\mathcal{K}_{s,t}^-$-free unbalanced signed graphs. Since the cases where $s=1$ or $t=1$ do not conform to the definition of  $\mathcal{K}_{s,t}^-$, it follows that $s,t\geq 2$.  Wang and Lin \cite{16} have solved the case of $s=t=2$ since in this situation $\mathcal{K}_{2,2}^-$ is $\mathcal{C}_{4}^{-}$. This paper gives an answer for $s=t=3$ and  completely characterizes the corresponding extremal signed graphs. In Fig.\ref {Fig1}, we use dashed lines to represent negative edges and solid lines to represent positive edges. Let $\Gamma_{n,t}$ be the signed graph obtained from a copy of $K_{n-1}$ with vertex set $\{v_1,...,v_{n-1}\}$ by adding a new vertex $u$ and $t-1$ edges $uv_1,...,uv_{t-1}$, where $uv_1$ is the unique negative edge. The main result of this paper is as follows.
\begin{thm}\label{thm1}
 Let $\Gamma$ be a $\mathcal{K}_{3,3}^{-}$-free unbalanced signed graph of order $n$ $(n\geq 7)$. Then $\lambda_1(A(\Gamma))\leq n-2$, with equality if and only if $\Gamma$ is switching isomorphic to $\Gamma_{n,3}$. 
\end{thm}
\section{Preliminaries}\label{se2}
Let $M$ be a real symmetric matrix with block form $M= [M_{ij}]$, and $q_{ij}$ be the average row sum of $M_{ij}$. Let $Q=(q_{ij})$ be the quotient matrix of $M$. Furthermore, $Q$ is referred to as an equitable quotient matrix if every block $M_{ij}$ has a constant row sum. Let Spec$(Q)=\{\lambda_1^{[t_1]},...,\lambda_k^{[t_k]}\}$ be the spectrum of $Q$, where eigenvalue 
$\lambda_i$ has multiplicity $t_i$ for $1 \leq i \leq k$. Let $P_{Q}(\lambda)=det(\lambda I-Q)$ denote the characteristic polynomial of $Q$. The matrix $J_{r\times s}$ is the all-one matrix of size $r \times s$, and when $r=s$, it is denoted by $J_r$. Also, we use $j_k=(1,\ldots,1)^T\in R^k$.
\begin{lemma}\label{el}\cite{21}
	There are two kinds of eigenvalues of the real symmetric matrix $M$.

	(i) The eigenvalues match the eigenvalues of $Q$.
	
	(ii) The eigenvalues of $M$ not in Spec$(Q)$ are unchanged when $\alpha $J is add to block $M_{ij}$ for every $1\leq i,j \leq m$, where  $\alpha$ is any constant. Moreover, $\lambda_1(M)=\lambda_1(Q)$ when $M$ is irreducible and nonnegative.
\end{lemma}
\begin{lemma}\label{l1} \cite{k2}

	(i) $\lambda_1(A(\Gamma_{n,t}))$ is the largest root of $g_{n,t}(x)=0$, where
	\begin{center}
		$g_{n,t}(x)=x^3-(n-3)x^2-(n+t-3)x-t^2+(n+4)t-n-7$.
	\end{center}

	(ii) $n-2\leq\lambda_1(A(\Gamma_{n,t}))<n-1$, with left equality if and only if $t=3$.
\end{lemma}
 Let $U$ be the signed graph obtained from a copy of $K_{n-2}$ with vertex set $\{v_3,...,v_{n}\}$ by adding two new vertices $v_1$, $v_2$ and $7$ edges $v_1v_2,v_1v_3,v_1v_4,v_1v_5,v_2v_3,v_2v_4,v_2v_5$, where $v_1v_2$ is the unique negative edge. Let $Z_1$ be the signed graph obtained from two copies of $K_{n-3}$ with vertex set $\{v_3,v_4,v_5,v_7,...,v_{n}\}$ and $\{v_3,v_4,v_6,...,v_{n}\}$ by adding two new vertices $v_1$, $v_2$ and $7$ edges $v_1v_2,v_1v_3,v_1v_4,v_1v_5,v_2v_3,v_2v_4,v_2v_6$, where $v_1v_2$ is the unique negative edge.  Let $Z_2$ be the signed graph obtained from two copies of $K_{n-3}$ with vertex set $\{v_3,v_5,...,v_{n}\}$ and $\{v_3,v_4,v_7,...,v_{n}\}$ by adding two new vertices $v_1$, $v_2$ and $7$ edges $v_1v_2,v_1v_3,v_1v_4,v_1v_5,v_2v_3,v_2v_4,v_2v_6$, where $v_1v_2$ is the unique negative edge. Let $W$ be the signed graph obtained from a copy of $K_{n-2}$ with vertex set $\{v_1,v_3,v_5,...,v_{n}\}$ by adding two new vertices $v_2$, $v_4$ and $6$ edges $v_2v_1,v_2v_3,v_2v_4,v_4v_1,v_4v_3,v_4v_5$, where $v_1v_2$ is the unique negative edge.  The above four graphs are shown in Fig.\ref {Fig1}.
\begin{lemma}\label{ll3}
	Let $n\geq 7$ be a positive integer and $\Gamma_{n,3}, U, Z_1, Z_2, W$ be the graphs depicted in Fig.\ref {Fig1}. Then 
\begin{center}
	$\lambda_1(A(\Gamma_{n,3}))>\max\{\lambda_1(A(U)),\lambda_1(A(Z_1)),\lambda_1(A(Z_2)),\lambda_1(A(W))\}$.
\end{center}
\end{lemma}
\begin{proof}
 We give $A(U)$ and its corresponding quotient matrix $Q_1$ by the vertex partition $V_1=\{v_1\}$, $V_2=\{v_2\}$, $V_3=\{v_3,v_4,v_5\}$ and $V_4=\{v_6,...,v_n\}$ as follows
 \begin{center}
 $A(U)= \begin{bmatrix} 0 & -1 & j_3^T & \bf{0}^T \\ -1 & 0 & j_3^T &  \bf{0}^T \\ j_3 & j_3 & (J-I)_3 & J_{3\times(n-5)} \\  \bf{0} &  \bf{0} & J_{(n-5)\times3} & (J-I)_{n-5} \end{bmatrix}$ and $Q_1= \begin{bmatrix} 0 & -1 & 3 & 0 \\ -1 & 0 & 3 & 0 \\ 1 & 1 & 2 & n-5 \\ 0 & 0 & 3 & n-6 \end{bmatrix}$.
 \end{center}
Note that the characteristic polynomial of $Q_1$ is
\begin{center}
$P_{Q_1}(\lambda)=(\lambda-1)(\lambda^3+(5-n)\lambda^2+(1-2n)\lambda+5n-33)$.
\end{center}
 Adding $\alpha $J to the blocks of $A(U)$, where $\alpha $ is constant, then
\begin{center}
$A_1= \begin{bmatrix} 0 & 0 & \bf{0}^T & \bf{0^T} \\ 0 & 0 & \bf{0}^T & \bf{0}^T \\ \bf0 & \bf0 & -I_3 & \bf{0}^T \\ \bf{0} & \bf{0} & \bf0 & -I_{n-5} \end{bmatrix}$.
\end{center}
Since $\lambda_1(Q_1)>0$ and Spec$(A_1)$$=\begin{Bmatrix}-1^{[n-2]},0^{[2]}\end{Bmatrix}$,  $\lambda_1(A(U))=\lambda_1(Q_1)$. Let $P_{Q_{11}}(\lambda)=\lambda^3+(5-n)\lambda^2+(1-2n)\lambda+5n-33$. Then $P_{Q_{11}}^{\prime}(\lambda)=3\lambda^2+(10-2n)\lambda+1-2n$. Note that the maximal solution of $P_{Q_{11}}^{\prime}(\lambda)=0$ is $\frac{n-5+\sqrt{n^2-4n+22}}{3}<n-2$, and $P_{Q_{11}}(n-2)=n^2-2n-23>0$ for $n \geq 7$. Thus, $\lambda_1({Q_{1}})<n-2$. Note that $\lambda_1(A(\Gamma_{n,3}))=n-2$ by Lemma \ref{l1}. So, $\lambda_1(A(\Gamma_{n,3}))>\lambda_1(A(U))$.

Secondly, we define $A(Z_1)$ and its corresponding quotient matrix $Q_2$ based on the vertex partition $V_1=\{v_1\}$, $V_2=\{v_2\}$, $V_3=\{v_3,v_4\}$, $V_4=\{v_5\}$, $V_5=\{v_6\}$ and $V_6=\{v_7,...,v_n\}$ as follows
 \begin{center}
 $\begin{gathered} A(Z_1)= \begin{bmatrix} 0 & -1 & j_2^T & 1 & 0 & \bf{0}^T \\ -1 & 0 & j_2^T & 0 & 1 & \bf{0}^T \\ j_2 & j_2 & \left(J-I\right)_2 & j_2 & j_2 & J_{2\times(n-6)} \\ 1 & 0 & j_2^T & 0 & 0 & j_{n-6}^T \\ 0 & 1 & j_2^T & 0 & 0 & j_{n-6}^T \\ \bf0 & \bf0 & J_{2\times(n-6)}^T & j_{n-6} & j_{n-6} & (J-I)_{n-6} \end{bmatrix} \end{gathered}$,
 \end{center}
 and 
 \begin{center}
 $Q_2= \begin{bmatrix} 0 & -1 & 2 & 1 & 0 & 0 \\ -1 & 0 & 2 & 0 & 1 & 0 \\ 1 & 1 & 1 & 1 & 1 & n-6 \\ 1 & 0 & 2 & 0 & 0 & n-6 \\ 0 & 1 & 2 & 0 & 0 & n-6 \\ 0 & 0 & 2 & 1 & 1 & n-7 \end{bmatrix}$.
 \end{center}
 Note that the characteristic polynomial of $Q_2$ is
 \begin{center}
 $P_{Q_2}(\lambda)=(\lambda^2-\lambda-1)(\lambda^4+(7-n)\lambda^3+(14-4n)\lambda^2-21\lambda+7n-53)$.
\end{center}
 Adding $\alpha $J to the blocks of $A(Z_1)$, where $\alpha $ is constant, then 
 \begin{center}
 	$A_2= \begin{bmatrix} 0 & 0 & \bf{0}^T & 0 & 0 & \bf{0}^T \\ 0 & 0 & \bf{0}^T & 0 & 0 & \bf{0}^T \\ \bf0 & \bf0 & -I_2 & \bf0 & \bf0 & \bf0 \\ 0 & 0 & \bf{0}^T & 0 & 0 & \bf{0}^T \\ 0 & 0 & \bf{0}^T & 0 & 0 & \bf{0}^T \\ \bf0 & \bf0 & \bf{0}^T & \bf0 & \bf0 & -I_{n-6} \end{bmatrix}$.
 \end{center}
Since $\lambda_1(Q_2)>0$ and Spec$(A_2)=\{-1^{[n-4]},0^{[4]}\}$,  $\lambda_1(A(Z_1))=\lambda_1(Q_2)$. Let  $P_{Q_{22}}(\lambda)=\lambda^4+(7-n)\lambda^3+(14-4n)\lambda^2-21\lambda+7n-53$. Then $P_{Q_{22}}^{\prime}(\lambda)=4\lambda^3+(21-3n)\lambda^2+(28-8n)\lambda-21$, and $P_{Q_{22}}^{\prime\prime}(\lambda)=12\lambda^2+(42-6n)\lambda+28-8n$. Note that the maximal solution of $P_{Q_{22}}^{\prime\prime}(\lambda)=0$ is $\frac{3n-21+\sqrt{9n^2-30n+105}}{12}<n-2$, $P_{Q_{22}}^{\prime}(n-2)=n^3+n^2-4n-25>0$,  and $P_{Q_{22}}(n-2)=n^3-26n-9>0$ for $n\geq 7$. This indicates that $\lambda_1({Q_{2}})<n-2$. Clearly, $\lambda_1(A(\Gamma_{n,3}))=n-2$ by Lemma \ref{l1}. Thus, $\lambda_1(A(\Gamma_{n,3}))>\lambda_1(A(Z_1))$.

Next, we define $A(Z_2)$ and its corresponding quotient matrix $Q_3$ according to the vertex partition $V_1=\{v_1\}$, $V_2=\{v_2\}$, $V_3=\{v_3\}$, $V_4=\{v_4\}$, $V_5=\{v_5\}$, $V_6=\{v_6\}$ and $V_7=\{v_7,...,v_n\}$ as follows
\begin{center}
$\begin{gathered} A(Z_2)= \begin{bmatrix} 0 & -1 & 1 & 1 & 1 & 0 & \bf{0}^T \\ -1 & 0 & 1 & 1 & 0 & 1 & \bf{0}^T \\ 1 & 1 & 0 & 1 & 1 & 1 & j_{n-6}^T \\ 1 & 1 & 1 & 0 & 0 & 0 & j_{n-6}^T \\ 1 & 0 & 1 & 0 & 0 & 1 & j_{n-6}^T \\ 0 & 1 & 1 & 0 & 1 & 0 & j_{n-6}^T \\ \bf0 & \bf0 & j_{n-6} & j_{n-6} & j_{n-6} & j_{n-6} & (J-I)_{n-6} \end{bmatrix}, \end{gathered}$
\end{center}
 and 
 \begin{center}
 	$Q_3= \begin{bmatrix} 0 & -1 & 1 & 1 & 1 & 0 & 0 \\ -1 & 0 & 1 & 1 & 0 & 1 & 0 \\ 1 & 1 & 0 & 1 & 1 & 1 & n-6 \\ 1 & 1 & 1 & 0 & 0 & 0 & n-6 \\ 1 & 0 & 1 & 0 & 0 & 1 & n-6 \\ 0 & 1 & 1 & 0 & 1 & 0 & n-6 \\ 0 & 0 & 1 & 1 & 1 & 1 & n-7 \end{bmatrix}$.
 \end{center}
 Note that the characteristic polynomial of $Q_3$ is
 \begin{center}
 	$P_{Q_3}(\lambda)=(\lambda^2-2)(\lambda^5+(7-n)\lambda^4+(15-4n)\lambda^3+(n-21)\lambda^2+(6n-36)\lambda+18-2n)$.
 \end{center}
  Adding $\alpha $J to the blocks of $A(Z_2)$, where $\alpha $ is constant, then 
  \begin{center}
  	$A_3= \begin{bmatrix} 0 & 0 & 0 & 0 & 0 & 0 & \bf{0}^T \\ 0 & 0 & 0 & 0 & 0 & 0 & \bf{0}^T \\ 0 & 0 & 0 & 0 & 0 & 0 & \bf{0}^T \\ 0 & 0 & 0 & 0 & 0 & 0 & \bf{0}^T \\ 0 & 0 & 0 & 0 & 0 & 0 & \bf{0}^T \\ 0 & 0 & 0 & 0 & 0 & 0 & \bf{0}^T \\ \bf0 & \bf0 & \bf0 & \bf0 & \bf0 & \bf0 & -I_{n-6} \end{bmatrix}$.
  \end{center}
 Since $\lambda_1(Q_3)>0$ and Spec$(A_3)=\{-1^{[n-6]},0^{[6]}\}$,  $\lambda_1(A(Z_2))=\lambda_1(Q_3)$. Let $P_{Q_{33}}(\lambda)=\lambda^5+(7-n)\lambda^4+(15-4n)\lambda^3+(n-21)\lambda^2+(6n-36)\lambda+18-2n$. Then $P_{Q_{33}}^{\prime}(\lambda)=5\lambda^4+(28-4n)\lambda^3+(45-12n)\lambda^2+(2n-42)\lambda+6n-36$, $P_{Q_{33}}^{\prime\prime}(\lambda)=20\lambda^3+(84-12n)\lambda^2+(90-24n)\lambda+2n-42$, and $P_{Q_{33}}^{\prime\prime\prime}(\lambda)=60\lambda^2+(168-24n)\lambda+90-24n$. Note that the maximal solution of $P_{Q_{33}}^{\prime\prime\prime}(\lambda)=0$ is $\frac{2n-14+\sqrt{4n^2-16n+46}}{10} <n-2$, $P_{Q_{33}}^{\prime\prime}(n-2)=8n^3-12n^2-4n-46>0$, $P_{Q_{33}}^{\prime}\left(n-2\right)=n^4-n^2-60n+84>0$, and $P_{Q_{33}}\left(n-2\right)=n^4-37n^2+90n-34>0$ for $n\geq 7$. This implies that $\lambda_1({Q_{3}})<n-2$. Obviously, $\lambda_1(A(\Gamma_{n,3}))=n-2$ by Lemma \ref{l1}. Thus, $\lambda_1(A(\Gamma_{n,3}))>\lambda_1(A(Z_2))$.

 Finally, we give $A(W)$ and its corresponding quotient matrix $Q_4$ by the vertex partition $V_1=\{v_1\}$, $V_2=\{v_2\}$, $V_3=\{v_3\}$, $V_4=\{v_4\}$, $V_5=\{v_5\}$, and $V_6=\{v_6,...,v_n\}$ as follows
 \begin{center}
$A(W)= \begin{bmatrix} 0 & -1 & 1 & 1 & 1 & j_{n-5}^T \\ -1 & 0 & 1 & 1 & 0 & \bf{0}^T \\ 1 & 1 & 0 & 1 & 1 & j_{n-5}^T \\ 1 & 1 & 1 & 0 & 1 & \bf{0}^T \\ 1 & 0 & 1 & 1 & 0 & j_{n-5}^T \\ j_{n-5} & \bf0 & j_{n-5} & \bf0 & j_{n-5} & (J-I)_{n-5} \end{bmatrix}$ and $Q_4= \begin{bmatrix} 0 & -1 & 1 & 1 & 1 & n-5 \\ -1 & 0 & 1 & 1 & 0 & 0 \\ 1 & 1 & 0 & 1 & 1 & n-5 \\ 1 & 1 & 1 & 0 & 1 & 0 \\ 1 & 0 & 1 & 1 & 0 & n-5 \\ 1 & 0 & 1 & 0 & 1 & n-6 \end{bmatrix}$.
 \end{center}
 Note that the characteristic polynomial of $Q_4$ is
 \begin{center}
 $P_{Q_4}(\lambda)=(\lambda+1)(\lambda^5+(5-n)\lambda^4+(1-2n)\lambda^3+(5n-31)\lambda^2+(7n-25)\lambda+33-5n)$.
 \end{center}
 Adding $\alpha $J to the blocks of $A(W)$, where $\alpha $ is constant, then 
 \begin{center}
 $A_4= \begin{bmatrix} 0 & 0 & 0 & 0 & 0 & \bf{0}^T \\ 0 & 0 & 0 & 0 & 0 & \bf{0}^T \\ 0 & 0 & 0 & 0 & 0 & \bf{0}^T \\ 0 & 0 & 0 & 0 & 0 & \bf{0}^T \\ 0 & 0 & 0 & 0 & 0 & \bf{0}^T \\ \bf0 & \bf0 & \bf0 & \bf0 & \bf0 & -I_{n-5} \end{bmatrix}$.
 \end{center}
  Since $\lambda_1(Q_4)>0$ and Spec$(A_4)=\{-1^{[n-5]},0^{[5]}\}$,  $\lambda_1(A(W))=\lambda_1(Q_4)$. Let $P_{Q_{44}}(\lambda)=\lambda^5+(5-n)\lambda^4+(1-2n)\lambda^3+(5n-31)\lambda^2+(7n-25)\lambda+33-5n$. Then $P_{Q_{44}}^{\prime}(\lambda)=5\lambda^4+(20-4n)\lambda^3+(3-6n)\lambda^2+(10n-62)\lambda+7n-25$, $P_{Q_{44}}^{\prime\prime}(\lambda)=20\lambda^3+(60-12n)\lambda^2+(6-12n)\lambda+10n-62$, and $P_{Q_{44}}^{\prime\prime\prime}(\lambda)=60\lambda^2+(120-24n)\lambda+6-12n$. Note that the maximal solution of $P_{Q_{44}}^{\prime\prime\prime}(\lambda)=0$ is $\frac{2n-10+\sqrt{4n^2-20n+90}}{10}<n-2$, $P_{Q_{44}}^{\prime\prime}(n-2)=8n^3-24n^2-8n+6>0$, $P_{Q_{44}}^{\prime}(n-2)=n^4-2n^3-11n^2+n+31>0$, and $P_{Q_{44}}(n-2)=n^4-6n^3-2n^2+32n-1>0$ for $n\geq 7$. This means that $\lambda_1({Q_{4}})<n-2$. Clearly, $\lambda_1(A(\Gamma_{n,3}))=n-2$ by Lemma \ref{l1}. Therefore, $\lambda_1(A(\Gamma_{n,3}))>\lambda_1(A(W))$. The proof is completed.
\end{proof}

\section{Proof of Theorem \ref{thm1}}
Let $\Gamma$ be a signed graph. The degree of a vertex $v_i$ in $\Gamma$ is denoted by $d_\Gamma(v_i)$ which is the number of edges incident with $v_i$.  We denote the set of all neighbors of $u$ in $\Gamma$ by $N_\Gamma(u)$ and $N_\Gamma[u]=N_\Gamma(u)\cup\{u\}$. Let $\rho(\Gamma)=\max\{|\lambda_i(\Gamma)|{:}1\leq i\leq n\}$ be the spectral radius of $\Gamma$. For $\phi\neq U\subset V(\Gamma)$, let $\Gamma[U]$ be the signed subgraph of $\Gamma$ induced by $U$. Let $\Gamma+uv$ (or $\Gamma-uv$) denote the signed graph obtained from  $\Gamma$ by adding (or deleting) the positive edge $uv$, where $u,v\in V(\Gamma)$. If all edges of $K_n$ are positive, then  we denote the graph by $(K_n,+)$. Let $K_n\circ K_1$ be a graph  obtained by taking one copy of $K_n$ and $n$ copies of $K_1$ and then forming a positive edge from $i^{th}$ vertex of $K_n$ to the vertex of the $i^{th}$ copy of $K_1$ for all $i$.

\begin{lemma}\label{l2}\cite{S1}
	Let $\Gamma$ be a signed graph. Then there exists a signed graph $\Gamma^{\prime}$ switching equivalent to $\Gamma$ such that $A(\Gamma^\prime)$ has a non-negative eigenvector corresponding to $\lambda_1(A(\Gamma^{\prime}))$.
\end{lemma}
\begin{lemma}\label{x1}\cite{15} 
Let $\Gamma=(G,\sigma)$ be a connected unbalanced signed graph of order $n$. If $\Gamma$ is $\mathcal{C}^-_3$-free, then $\rho(\Gamma)\leq \frac{1}{2}(\sqrt{n^2-8}+n-4)$.
\end{lemma}
\begin{lemma}\label{l3}\cite{7}
If signed graph $\Gamma=(G,\sigma)$ with $n$ vertices $(n \geq 7)$ is unbalanced and does not contain unbalanced ${K}_4$ as a signed subgraph, then $\rho(\Gamma)\leq n-2$, with equation holds only when $\Gamma$ is switching isomorphic to $\Gamma_{n,3}$.
\end{lemma}
\begin{lemma}\label{l5}\cite{B1}
Let $X=(x_1,x_2,...,x_n)^T$ be an eigenvector associated with the index of a signed graph $\Gamma$ and let $v_r$, $v_s$ be fixed vertices of $\Gamma$.

(i) If $x_rx_s\geq 0$, at least one of $x_r,x_s$ is nonzero, and $v_r$ and $v_s$ are not adjacent (resp. $v_rv_s$ is a negative edge), then for a signed graph $\Gamma^\prime$ obtained by adding a positive edge $v_rv_s$ (resp. removing $v_rv_s$ or reversing its sign), we have $\lambda_1(A(\Gamma^\prime))>\lambda_1(A(\Gamma))$.

(ii) If $x_r\geq x_s$, $w\in N_{\Gamma}(v_s)\backslash N_{\Gamma}(v_r)$, and $x_w>0$, then for a signed graph $\Gamma^\prime$ obtained by moving positive edge $v_sw$ from $v_s$ to $v_r$, we have $\lambda_1(A(\Gamma^\prime))>\lambda_1(A(\Gamma))$.
\end{lemma}

\noindent{\bf{Proof of Theorem \ref{thm1}}.} Let $\Gamma=(G,\sigma)$ be a  $\mathcal{K}_{3,3}^{-}$-free unbalanced signed graph on $n\geq 7$ vertices with maximum index. According to Lemma \ref{l2}, $\Gamma$ is switching equivalent to a signed graph $\Gamma^{\prime}$ such that $A(\Gamma^\prime)$ has a non-negative eigenvector corresponding to $\lambda_1(A(\Gamma^\prime))=\lambda_1(A(\Gamma))$. Note that $\Gamma$ and $\Gamma^\prime$ share the same positive and negative cycles. So, $\Gamma^\prime$ is unbalanced and $\mathcal{K}_{3,3}^{-}$-free. Let $V(\Gamma)=\{v_1,v_2,...,v_n\}$ and $X=(x_1,x_2,...,x_n)^T$ be the non-negative unit eigenvector of $A(\Gamma^\prime)$ corresponding to $\lambda_1(A(\Gamma^\prime))$. Note that $\Gamma_{n,3}$ is unbalanced and $\mathcal{K}_{3,3}^{-}$-free. By Lemma \ref{l1}, $\lambda_1(A(\Gamma^\prime))\geq \lambda_1(A(\Gamma_{n,3}))=n-2$. Since $\frac{1}{2}(\sqrt{n^2-8}+n-4) <n-2$, $\Gamma^\prime$ must contain an unbalanced $C_3$ as a signed subgraph by Lemma \ref{x1}. Assume that $C_3$ is an unbalanced signed subgraph of $\Gamma^{\prime}$ and $V(C_3)=\{v_1,v_2,v_3\}$.
\begin{claim}\label{c1}
$X$ contains at most one zero entry.
\end{claim}
\begin{proof}
Otherwise, $X$ contains at least two zero entries. Assume that $x_n=x_{n-1}=0$, then 
\begin{align*}
\lambda_1(A(\Gamma^{\prime}))&=X^TA(\Gamma^{\prime})X=(x_1,\ldots,x_{n-2})A(\Gamma^{\prime}-v_n-v_{n-1})(x_1,\ldots,x_{n-2})^T\\&\leq\lambda_1(A(\Gamma^{\prime}-v_n-v_{n-1}))\leq\lambda_1(A(K_{n-2}))=n-3<\lambda_1(A(\Gamma^{\prime})),
\end{align*}
a contradiction. Thus, $X$ contains at most one zero entry.
\end{proof} 
\begin{claim}\label{c2}
The unbalanced $C_3$ contains all negative edges of $\Gamma^\prime$.
\end{claim}
\begin{proof}
Otherwise, suppose that there is a negative edge $v_iv_j$ of $\Gamma^\prime$ such that $v_iv_j\notin E(C_3)$. Then we can construct a new unbalanced signed graph $\Gamma^{\prime\prime}$ by removing the negative edge $v_{i}v_{j}$ such that $\Gamma^{\prime\prime}$ is a $\mathcal{K}_{3,3}^{-}$-free unbalanced signed graph and $\lambda_1(A(\Gamma^{\prime\prime}))>\lambda_1(A(\Gamma^{\prime}))$ by $(i)$ of Lemma \ref{l5}. This contradicts the maximality of $\lambda_1(A(\Gamma^\prime))$. Thus, Claim \ref{c2} holds.
\end{proof}
 Assume that $k$ is the smallest positive integer such that $x_{k}=\max_{1\leq i\leq n}x_{i}$. By Claim \ref{c1}, $x_k>0$ clearly.
\begin{claim}\label{o3}
The unbalanced $C_3$ contains exactly one negative edge.
\end{claim}
\begin{proof}
Otherwise, the unbalanced $C_3$ contains three negative edges of $\Gamma^\prime$.  Note that there is at most one zero entry of $X$ by Claim \ref{c1}. If $k\leq 3$, then
\begin{align*}
\lambda_1(A(\Gamma^{\prime}))x_k&=-(x_1+x_2+x_3)+x_k+\sum_{v_i\in N_{\Gamma^\prime}(v_k)\setminus V(C_3)}x_i\\&\leq-(x_1+x_2+x_3)+x_k+(n-3)x_k\\&<(n-3)x_k.
\end{align*}
This implies that $\lambda_1(A(\Gamma^{\prime}))< n-3$, a contradiction. Thus, $k>4$. And then
\begin{center}
$(n-2)x_k\leq\lambda_1(A(\Gamma^{\prime}))x_k=\sum\limits_{v_i\in N_{\Gamma^\prime}(v_k)}x_i\leq d_{\Gamma^{\prime}}(v_k)x_k$,
\end{center}
that is, $d_{\Gamma^{\prime}}(v_k)=n-2$ or $n-1$. If $d_{\Gamma^{\prime}}(v_k)=n-2$, then $x_i=x_k$ for any $v_i\in N_{\Gamma^{\prime}}(v_k)$. It means that at least one of $x_i$ with $i=1,2,3$ is equal to $x_k$, contradicting the choice of $k$. Thus, $d_{\Gamma^{\prime}}(v_k)=n-1$. Now, we can construct a new unbalanced signed graph $\Gamma^{\prime\prime}$ by removing the negative edge $v_{1}v_{2}$ such that $\Gamma^{\prime\prime}$ is  still a $\mathcal{K}_{3,3}^{-}$-free unbalanced signed graph but $\lambda_1(A(\Gamma^{\prime\prime}))>\lambda_1(A(\Gamma^{\prime}))$ by $(i)$ of Lemma \ref{l5}. This contradicts the maximality of $\lambda_1(A(\Gamma^\prime))$. So, the unbalanced $C_3$ contains exactly one negative edge.
\end{proof}
 Claims \ref{c2} and \ref{o3} show that $\Gamma^{\prime}$ contains only one negative edge, and it is the negative edge of the unbalanced $C_3$. Assume that this edge is $v_1v_2$.
\begin{claim}\label{CC1}
If $X>0$, then $k \geq 3$ and $d_{\Gamma^{\prime}}(v_k)=n-1$.
\end{claim}
\begin{proof}
If $k<3$, then $(n-2)x_k\leq\lambda_1(A(\Gamma^{\prime}))x_k\leq-x_{3-k}+(n-2)x_k<(n-2)x_k$, a contradiction. Thus, $k \geq 3$. Note that
\begin{center}
$(n-2)x_k\leq\lambda_1(A(\Gamma^{\prime}))x_k=\sum\limits_{v_i\in N_{\Gamma^\prime}(v_k)}x_i\leq d_{\Gamma^{\prime}}(v_k)x_k$,
\end{center}
then $d_{\Gamma^{\prime}}(v_k) \geq n-2$. If $d_{\Gamma^{\prime}}(v_k)=n-2$, then the entry of $X$ corresponding to each neighbor of $v_k$ equals $x_k$. Note that one of $v_1$ and $v_2$ is adjacent to $v_k$. Without loss of generality, assume that $x_1=x_k$, then $(n-2)x_k\leq\lambda_1(A(\Gamma^{\prime}))x_k=\lambda_1(A(\Gamma^{\prime}))x_1\leq-x_2+(d_{\Gamma^{\prime}}(v_1)-1)x_k<(n-2)x_k$, a contradiction. Hence,  $d_{\Gamma^{\prime}}(v_k)=n-1$.
\end{proof}
Next, we divide the proof into the following two cases.
\begin{case}\label{Case1}
There exists an integer $r$ such that $x_r=0$ for $1\leq r \leq n$.
\end{case}
Firstly, we assert that $d_{\Gamma^{\prime}}(v_r)\geq 1$. Otherwise, $d_{\Gamma^{\prime}}(v_r)=0$. Let $\Gamma^{\prime\prime}=\Gamma^{\prime}+v_1v_r$, then $\Gamma^{\prime\prime}$ is a $\mathcal{K}_{3,3}^{-}$-free unbalanced signed graph and $\lambda_1(A(\Gamma^{\prime\prime}))>\lambda_1(A(\Gamma^{\prime}))$ by $(i)$ of Lemma \ref{l5}, this contradicts the maximality of $\lambda_1(A(\Gamma^\prime))$. Thus, $d_{\Gamma^{\prime}}(v_r)\geq 1$. If $r\geq 3$, then $0=\lambda_1(A(\Gamma^{\prime}))x_r=\sum_{v_i\in N_{\Gamma^{\prime}}(v_r)}x_i>0$, a contradiction. Thus, $r=1$ or $2$. Without loss of generality, assume that $r=1$. Then  $k\geq 2$. Note that
\begin{center}
	$(n-2)x_k\leq\lambda_1(A(\Gamma^{\prime}))x_k=\sum\limits_{v_i\in N_{\Gamma^\prime}(v_k)}x_i\leq d_{\Gamma^{\prime}}(v_k)x_k$,
\end{center}
then  $d_{\Gamma^{\prime}}(v_k) \geq n-2$. If  $d_{\Gamma^{\prime}}(v_k)=n-2$, then each of the $n-2$ entries  of $X$ corresponding to the neighbors of $v_k$ is equal to $x_k$. It implies that $x_2=\cdots=x_n$. If $d_{\Gamma^{\prime}}(v_k)=n-1$, then
\begin{center}
$(n-2)x_{{k}}\leq\lambda_{1}(A(\Gamma^{\prime}))x_{{k}}=x_{1}+\sum\limits_{v_{i}\in N_{\Gamma^{\prime}}(v_{k})\setminus\{v_{1}\}}x_{i}\leq(d_{\Gamma^{\prime}}(v_{{k}})-1)x_{{k}}=(n-2)x_k$.
\end{center}
Consequently, $x_2=\cdots=x_n$. This means that $d_{\Gamma^{\prime}}(v_i)=n-2$ or $n-1$ and $v_i$ is adjacent to all other vertices $V(\Gamma^{\prime})\backslash\{v_1\}$ for any $i\in[2,n]$. Therefore, $\Gamma^{\prime}[V(\Gamma^{\prime})\backslash\{v_1\}]\cong(K_{n-1},+)$. If there exists an integer $i$ such that $d_{\Gamma^{\prime}}(v_i)=n-1$ for $i\in[4,n]$, then $\Gamma^{\prime}$ contains an unbalanced ${K}_{3,3}$, a contradiction. Therefore,  $\Gamma^{\prime}$ is switching isomorphic to $\Gamma_{n,3}$ and $\lambda_1(A(\Gamma^\prime))=n-2$.
\begin{case}
$X>0$.
\end{case}
By Claim \ref{CC1}, $k \geq 3$ and $d_{\Gamma^{\prime}}(v_k)=n-1$. Without loss of generality, assume that $k=3$ and $d_{\Gamma^{\prime}}(v_1)\geq d_{\Gamma^{\prime}}(v_2)$.
If $\Gamma^{\prime}$ does not contain an unbalanced ${K}_4$ as a signed subgraph, then $\Gamma^{\prime}$ is switching isomorphic to $\Gamma_{n,3}$ and $\lambda_1(A(\Gamma^\prime))=n-2$ by Lemma \ref{l3}. Next, we assume that $\Gamma^{\prime}$ contains an unbalanced ${K}_4$ as a signed subgraph. From Claims \ref{c2} and \ref{o3}, we may assume that $V(K_4)=\{v_1,v_2,v_3,v_4\}$. After the above preparations, we will further discuss in six subcases.
\begin{subcase}
$d_{\Gamma^{\prime}}(v_1)= d_{\Gamma^{\prime}}(v_2)=3$, i.e., $N_{\Gamma^{\prime}}[v_1]=N_{\Gamma^{\prime}}[v_2]=\{v_1,v_2,v_3,v_4\}$.
\end{subcase}
Obviously, $\Gamma^{\prime}[V(\Gamma^{\prime})\backslash\{v_1,v_2\}]\cong(K_{n-2},+)$ by $(i)$ of Lemma \ref{l5}.  Note that $U$ is a $\mathcal{K}_{3,3}^{-}$-free unbalanced signed graph. Thus, $\Gamma^{\prime}$ is switching isomorphic to $U-v_1v_5-v_2v_5$.  However, $\lambda_1(A(U-v_1v_5-v_2v_5))<\lambda_1(A(U))$, this  contradicts the maximality of $\lambda_1(A(\Gamma^{\prime}))$.
\begin{subcase}
 $d_{\Gamma^{\prime}}(v_1)=4$ and $d_{\Gamma^{\prime}}(v_2)=3$, i.e., $N_{\Gamma^{\prime}}[v_2]=\{v_1,v_2,v_3,v_4\}\subset N_{\Gamma^{\prime}}[v_1]$.
\end{subcase}
 Without loss of generality, assume that $N_{\Gamma^{\prime}}(v_1)=\{v_2,v_3,v_4,v_5\}$. By $(i)$ of Lemma \ref{l5}, $\Gamma^{\prime}[V(\Gamma^{\prime})\backslash\{v_1,v_2\}]\cong(K_{n-2},+)$.  Note that $U$ is a $\mathcal{K}_{3,3}^{-}$-free unbalanced signed graph. So, $\Gamma^{\prime}$ is switching isomorphic to $U-v_2v_5$. However, $\lambda_1(A(U-v_2v_5))<\lambda_1(A(U))$, this also contradicts the maximality of $\lambda_1(A(\Gamma^{\prime}))$.
\begin{subcase}
$d_{\Gamma^{\prime}}(v_1)=d_{\Gamma^{\prime}}(v_2)=4$.
\end{subcase}
Without loss of generality, assume that $N_{\Gamma^{\prime}}(v_1)=\{v_2,v_3,v_4,v_5\}$. If $v_2v_5 \in E(\Gamma^{\prime})$, then $\Gamma^{\prime}[V(\Gamma^{\prime})\backslash\{v_1,v_2\}]\cong(K_{n-2},+)$ by $(i)$ of Lemma \ref{l5}. Therefore, $\Gamma^{\prime}$ is switching isomorphic to $U$. If $v_2v_5 \notin E(\Gamma^{\prime})$, then we assume that $N_{\Gamma^{\prime}}(v_2)=\{v_1,v_3,v_4,v_6\}$ by $d_{\Gamma^{\prime}}(v_2)=4$. Clearly, $v_6$ is adjacent to either $v_4$ or $v_5$ by $(i)$ of Lemma \ref{l5}. Otherwise, $\Gamma^{\prime}$ contains an unbalanced $K_{3,3}$, a contradiction. If $v_4v_6\in E(\Gamma^{\prime})$, by $(i)$ of Lemma \ref{l5}, then $\Gamma^{\prime}[V(\Gamma^{\prime})\backslash\{v_1,v_2\}]\cong(K_{n-2},+)-v_5v_6$. So, $\Gamma^{\prime}$ is switching isomorphic to $Z_1$. If $v_5v_6\in E(\Gamma^{\prime})$, by $(i)$ of Lemma \ref{l5}, then $\Gamma^{\prime}[V(\Gamma^{\prime})\backslash\{v_1,v_2\}]\cong(K_{n-2},+)-v_4v_5-v_4v_6$. Thus, $\Gamma^{\prime}$ is switching isomorphic to $Z_2$. However, $\lambda_1(A(\Gamma_{n,3}))>\max\{\lambda_1(A(U)),\lambda_1(A(Z_1)),\lambda_1(A(Z_2))\}$ by Lemma \ref{ll3}. This contradicts the maximality of $\lambda_1(A(\Gamma^{\prime}))$.
\begin{subcase}
$d_{\Gamma^{\prime}}(v_1)\geq 5$ and $d_{\Gamma^{\prime}}(v_2)=3$.
\end{subcase}
We first assert that $2\leq |N_{\Gamma^{\prime}}(v_1)\cap N_{\Gamma^{\prime}}(v_4)|\leq3$. Otherwise, $|N_{\Gamma^{\prime}}(v_1)\cap N_{\Gamma^{\prime}}(v_4)|\geq4$. Without loss of generality, assume that $\{v_{2},v_{3},v_{5},v_{6}\}\subseteq N_{\Gamma^{\prime}}(v_{1})\cap N_{\Gamma^{\prime}}(v_{4})$, then $\Gamma^{\prime}[N_{\Gamma^{\prime}}[v_1]]$ contains an unbalanced $K_{3,3}$, a contradiction. Next, we claim that $|N_{\Gamma^{\prime}}(v_1)\cap N_{\Gamma^{\prime}}(v_4)|=3$. Otherwise, $|N_{\Gamma^{\prime}}(v_1)\cap N_{\Gamma^{\prime}}(v_4)|=2$, i.e., $N_{\Gamma^{\prime}}(v_1)\cap N_{\Gamma^{\prime}}(v_4)=\{v_2,v_3\}$. Assume that $v_5 \in N_{\Gamma^{\prime}}(v_{1})$, let $\Gamma^{\prime\prime}=\Gamma^{\prime}+v_4v_5$, then $\Gamma^{\prime\prime}$ is a $\mathcal{K}_{3,3}^{-}$-free unbalanced signed graph and $\lambda_1(A(\Gamma^{\prime\prime}))>\lambda_1(A(\Gamma^{\prime}))$ by $(i)$ of Lemma \ref{l5}, a contradiction. Thus, $|N_{\Gamma^{\prime}}(v_1)\cap N_{\Gamma^{\prime}}(v_4)|=3$. Assume that $N_{\Gamma^{\prime}}(v_1)\cap N_{\Gamma^{\prime}}(v_4)=\{v_2, v_3, v_5\}$. Finally, we assert that $d_{\Gamma^{\prime}}(v_4)=4$. Otherwise, assume that $v_6\in N_{\Gamma^{\prime}}(v_1)$ and $v_7\in N_{\Gamma^{\prime}}(v_4)$. If $x_1\geq x_4$, let  $\Gamma^{\prime\prime}=\Gamma^{\prime}+v_1v_7-v_4v_7$, then $\Gamma^{\prime\prime}$ is a $\mathcal{K}_{3,3}^{-}$-free unbalanced signed graph and $\lambda_1(A(\Gamma^{\prime\prime}))>\lambda_1(A(\Gamma^{\prime}))$ by $(ii)$ of Lemma \ref{l5}, a contradiction.  If $x_1<x_4$, let  $\Gamma^{\prime\prime}=\Gamma^{\prime}+v_4v_6-v_1v_6$, then $\Gamma^{\prime\prime}$ is a $\mathcal{K}_{3,3}^{-}$-free unbalanced signed graph and $\lambda_1(A(\Gamma^{\prime\prime}))>\lambda_1(A(\Gamma^{\prime}))$ by $(ii)$ of Lemma \ref{l5}, a contradiction. Thus, $d_{\Gamma^{\prime}}(v_4)=4$. By $(i)$ of Lemma \ref{l5}, $\Gamma^{\prime}[V(\Gamma^{\prime})\backslash\{v_{2},v_{4}\}]\cong(K_{n-2},+)$. Thus, $\Gamma^{\prime}$ is switching isomorphic to $W$. However, $\lambda_1(A(\Gamma_{n,3}))>\lambda_1(A(W))$ by Lemma \ref{ll3}. This contradicts the maximality of $\lambda_1(A(\Gamma^{\prime}))$.

For convenience, we denote $\lambda_1(A(\Gamma^{\prime}))x_i$ by $\lambda_1x_i$ for all $x_i\in X$.
\begin{subcase}
$d_{\Gamma^{\prime}}(v_1)\geq 5$ and $d_{\Gamma^{\prime}}(v_2)=4$.
\end{subcase}
We first consider that $|N_{\Gamma^{\prime}}(v_1)\cap N_{\Gamma^{\prime}}(v_2)|=2$, i.e., $N_{\Gamma^{\prime}}(v_1)\cap N_{\Gamma^{\prime}}(v_2)=\{v_3,v_4\}$. Assume that $v_6\in N_{\Gamma^{\prime}}(v_2)$ by $d_{\Gamma^{\prime}}(v_2)=4$. Now, we assert that $|N_{\Gamma^{\prime}}(v_1)\cap N_{\Gamma^{\prime}}(v_6)|=3$. Otherwise, $|N_{\Gamma^{\prime}}(v_1)\cap N_{\Gamma^{\prime}}(v_6)|\neq 3$. If $|N_{\Gamma^{\prime}}(v_1)\cap N_{\Gamma^{\prime}}(v_6)|=2$, i.e., $N_{\Gamma^{\prime}}(v_1)\cap N_{\Gamma^{\prime}}(v_6)=\{v_2,v_3\}$, let $\Gamma^{\prime\prime}=\Gamma^{\prime}+v_4v_6$, then $\Gamma^{\prime\prime}$ is a $\mathcal{K}_{3,3}^{-}$-free unbalanced signed graph and $\lambda_1(A(\Gamma^{\prime\prime}))>\lambda_1(A(\Gamma^{\prime}))$ by $(i)$ of Lemma \ref{l5}, a contradiction. If $|N_{\Gamma^{\prime}}(v_1)\cap N_{\Gamma^{\prime}}(v_6)|\geq 4$, then $\Gamma^{\prime}$ must contain an unbalanced $K_{3,3}$, a contradiction. Thus, $|N_{\Gamma^{\prime}}(v_1)\cap N_{\Gamma^{\prime}}(v_6)|=3$, i.e., $N_{\Gamma^{\prime}}(v_1)\cap N_{\Gamma^{\prime}}(v_6)=\{v_2,v_3,u\}$. Next, we will divide it into two cases. If $u=v_4$, by $(i)$ of Lemma \ref{l5}, then $\Gamma^{\prime}[N_{\Gamma^{\prime}}(v_1)\backslash\{v_2,v_4\}]\cong (K_{d_{\Gamma^{\prime}}(v_1)-2},+)$ and $v_i$ is adjacent to every vertex in $V(\Gamma^{\prime})\backslash\{v_1,v_2\}$ for all $v_i\in V(\Gamma^{\prime})\backslash(N_{\Gamma^{\prime}}[v_1]\cup N_{\Gamma^{\prime}}[v_2])$. We first claim that $|N_{\Gamma^{\prime}}(v_1)\cap N_{\Gamma^{\prime}}(v_4)|=3$. Otherwise, $|N_{\Gamma^{\prime}}(v_1)\cap N_{\Gamma^{\prime}}(v_4)|\neq 3$. If $|N_{\Gamma^{\prime}}(v_1)\cap N_{\Gamma^{\prime}}(v_4)|=2$, i.e., $N_{\Gamma^{\prime}}(v_1)\cap N_{\Gamma^{\prime}}(v_4)=\{v_2,v_3\}$, let $\Gamma^{\prime\prime}=\Gamma^{\prime}+v_4v_5$, where $v_5\in N_{\Gamma^{\prime}}(v_1)$, then $\Gamma^{\prime\prime}$ is a $\mathcal{K}_{3,3}^{-}$-free unbalanced signed graph and $\lambda_1(A(\Gamma^{\prime\prime}))>\lambda_1(A(\Gamma^{\prime}))$ by $(i)$ of Lemma \ref{l5}, a contradiction. If $|N_{\Gamma^{\prime}}(v_1)\cap N_{\Gamma^{\prime}}(v_4)|\geq 4$, then $\Gamma^{\prime}$ contains an unbalanced $K_{3,3}$, a contradiction. Thus, $|N_{\Gamma^{\prime}}(v_1)\cap N_{\Gamma^{\prime}}(v_4)|=3$. Without loss of generality, assume that $N_{\Gamma^{\prime}}(v_1)\cap N_{\Gamma^{\prime}}(v_4)=\{v_2,v_3,v_5\}$. Note that $\lambda_1x_1=\sum_{v_i\in N_{\Gamma^{\prime}}(v_1)\setminus\{v_2,v_3,v_4\}}x_i-x_2+x_3+x_4$,
$\lambda_1x_2=-x_1+x_3+x_4+x_6$. Then $\lambda_1(x_1-x_2)=\sum_{v_i\in N_{\Gamma^{\prime}}(v_1)\setminus\{v_2,v_3,v_4\}}x_i+x_1-x_2-x_6$, that is, $(\lambda_1-1)(x_1-x_2)=\sum_{v_i\in N_{\Gamma^{\prime}}(v_1)\setminus\{v_2,v_3,v_4\}}x_i-x_6$. It is evident that $\lambda_1(\sum_{v_i\in N_{\Gamma^{\prime}}(v_1)\setminus\{v_2,v_3,v_4\}}x_i)>2x_3+x_4+\sum_{v_j\in V(\Gamma^{\prime})\setminus(N_{\Gamma^{\prime}}[v_1]\cup N_{\Gamma^{\prime}}[v_2])}x_j$ and  $\lambda_1x_6=x_2+x_3+x_4+$ $\sum_{v_j\in V(\Gamma^{\prime})\setminus(N_{\Gamma^{\prime}}[v_1]\cup N_{\Gamma^{\prime}}[v_2])}x_j$. Thus,
$\lambda_1(\sum_{v_i\in N_{\Gamma^{\prime}}(v_1)\setminus\{v_2,v_3,v_4\}}x_i-x_6)>x_3-x_2>0$ and $x_1>x_2$. Let $\Gamma^{\prime\prime}=\Gamma^{\prime}+v_1v_6+v_6w-v_2v_6-v_4v_6$ for all $w\in N_{\Gamma^{\prime}}(v_1)\setminus\{v_2,v_3,v_4\}$, then $\Gamma^{\prime\prime}$ is a $\mathcal{K}_{3,3}^{-}$-free unbalanced signed graph. Note that $\lambda_1(\sum_{w\in N_{\Gamma^{\prime}}(v_1)\setminus\{v_2,v_3,v_4\}}x_w)>2x_1+2x_3+x_5+\sum_{v_j\in V(\Gamma^{\prime})\setminus(N_{\Gamma^{\prime}}[v_1]\cup N_{\Gamma^{\prime}}[v_2])}x_j$, $\lambda_1x_4=x_1+x_2+x_3+x_5+x_6+\sum_{v_j\in V(\Gamma^{\prime})\setminus(N_{\Gamma^{\prime}}[v_1]\cup N_{\Gamma^{\prime}}[v_2])}x_j$.  Since $x_1>x_2$ and $x_3> x_6$, $\lambda_1(\sum_{w\in N_{\Gamma^{\prime}}(v_1)\setminus\{v_2,v_3,v_4\}}x_w-x_4)>x_1+x_3-x_2-x_6>0$. Thus,
\begin{align*}
	\lambda_1(A(\Gamma^{\prime\prime}))-\lambda_1(A(\Gamma^{\prime}))&\geq X^T(A(\Gamma^{\prime\prime})-A(\Gamma^{\prime}))X\\
	&=2x_6(\sum_{w\in N_{\Gamma^{\prime}}(v_1)\setminus\{v_2,v_3,v_4\}}x_w-x_4+x_1-x_2)\\&>0,
\end{align*} 
  a contradiction. If $u\neq v_4$, by $(i)$ of Lemma \ref{l5}, then $\Gamma^{\prime}[N_{\Gamma^{\prime}}(v_1)\backslash\{v_2,v_4\}]\cong (K_{d_{\Gamma^{\prime}}(v_1)-2},+)$ and $v_i$ is adjacent to every vertex in $V(\Gamma^{\prime})\backslash\{v_1,v_2\}$ for all $v_i\in V(\Gamma^{\prime})\backslash(N_{\Gamma^{\prime}}[v_1]\cup N_{\Gamma^{\prime}}[v_2])$. Similarly,  $x_1>x_2$. Let $\Gamma^{\prime\prime}=\Gamma^{\prime}+v_1v_6-v_2v_6$, then $\Gamma^{\prime\prime}$ is a $\mathcal{K}_{3,3}^{-}$-free unbalanced signed graph and $\lambda_1(A(\Gamma^{\prime\prime}))>\lambda_1(A(\Gamma^{\prime}))$ by $(ii)$ of Lemma \ref{l5}, a contradiction.

Next, we assume that $|N_{\Gamma^{\prime}}(v_1)\cap N_{\Gamma^{\prime}}(v_2)|=3, i.e., N_{\Gamma^{\prime}}(v_1)\cap N_{\Gamma^{\prime}}(v_2)=\{v_3,v_4,v_5\}$. We first assert that $2\leq |N_{\Gamma^{\prime}}(v_1)\cap N_{\Gamma^{\prime}}(v_4)|\leq 3$. Otherwise, $|N_{\Gamma^{\prime}}(v_1)\cap N_{\Gamma^{\prime}}(v_4)|\geq 4$, then $\Gamma^{\prime}$ contains an unbalanced $K_{3,3}$, a contradiction. Assume that $v_6\in N_{\Gamma^{\prime}}(v_1)$ by $d_{\Gamma^{\prime}}(v_1)\geq 5$. Now, we will divide into the following three cases. 

$(1)$ $N_{\Gamma^{\prime}}(v_1)\cap N_{\Gamma^{\prime}}(v_4)=\{v_2,v_3,v_5\}$, then $v_5u \notin E(\Gamma^{\prime})$ for all $u\in N_{\Gamma^{\prime}}(v_1)\setminus\{v_2,v_3,v_4,v_5\}$ since $\Gamma^{\prime}$ is a $\mathcal{K}_{3,3}^{-}$-free unbalanced signed graph. By $(i)$ of Lemma \ref{l5}, $\Gamma^{\prime}[N_{\Gamma^{\prime}}(v_1)\backslash\{v_2,v_4,v_5\}]\cong( K_{d_{\Gamma^{\prime}}(v_1)-3},+)$. If $5\leq d_{\Gamma^{\prime}}(v_1)\leq 6$, let $\Gamma^{\prime\prime}=\Gamma^{\prime}+v_4v_6+v_5v_6-v_1v_6$, then $\Gamma^{\prime\prime}$ is a $\mathcal{K}_{3,3}^{-}$-free unbalanced signed graph. Note that $\lambda_1(x_4+x_5)\geq 2x_1+2x_2+2x_3+x_4+x_5$, $\lambda_1x_1=-x_2+x_3+x_4+x_5+\sum_{v_i\in N_{\Gamma^{\prime}}(v_1)\setminus\{v_2,v_3,v_4,v_5\}}x_i$. Then $\lambda_1(x_4+x_5-x_1)\geq 2x_1+3x_2+x_3-\sum_{v_i\in N_{\Gamma^{\prime}}(v_1)\setminus\{v_2,v_3,v_4,v_5\}}x_i$. That is, $(\lambda_1+2)(x_4+x_5-x_1)\geq 2x_4+2x_5+3x_2+x_3-\sum_{v_i\in N_{\Gamma^{\prime}}(v_1)\setminus\{v_2,v_3,v_4,v_5\}}x_i$. It is evident that $2x_4+2x_5+x_3>\sum_{v_i\in N_{\Gamma^{\prime}}(v_1)\setminus\{v_2,v_3,v_4,v_5\}}x_i$. Thus, $x_4+x_5-x_1> 0$ and
\begin{align*}
	\lambda_1(A(\Gamma^{\prime\prime}))-\lambda_1(A(\Gamma^{\prime}))&\geq X^T(A(\Gamma^{\prime\prime})-A(\Gamma^{\prime}))X
	=2x_6(x_4+x_5-x_1)>0,
\end{align*} 
 a contradiction. If $d_{\Gamma^{\prime}}(v_1)\geq 7$, let $\Gamma^{\prime\prime}=\Gamma^{\prime}+v_5w-v_2v_5$ for all $w\in N_{\Gamma^{\prime}}(v_1)\setminus\{v_2,v_3,v_4,v_5\}$, then $\Gamma^{\prime\prime}$ is a $\mathcal{K}_{3,3}^{-}$-free unbalanced signed graph. Note that  $\lambda_1(\sum_{w\in N_{\Gamma^{\prime}}(v_1)\setminus\{v_2,v_3,v_4,v_5\}}x_w)$ $>3x_1+3x_3$, $\lambda_1x_2=-x_1+x_3+x_4+x_5$.  
 Then $\lambda_1(\sum_{w\in N_{\Gamma^{\prime}}(v_1)\setminus\{v_2,v_3,v_4,v_5\}}x_w-x_2)>4x_1+2x_3-x_4-x_5>0$ and 
\begin{align*}
	\lambda_1(A(\Gamma^{\prime\prime}))-\lambda_1(A(\Gamma^{\prime}))&\geq X^T(A(\Gamma^{\prime\prime})-A(\Gamma^{\prime}))X
	=2x_5(\sum_{w\in N_{\Gamma^{\prime}}(v_1)\setminus\{v_2,v_3,v_4,v_5\}}x_w-x_2)>0,
\end{align*} 
a contradiction.

$(2)$ $N_{\Gamma^{\prime}}(v_1)\cap N_{\Gamma^{\prime}}(v_4)=\{v_2,v_3,v_6\}$. By $(i)$ of Lemma \ref{l5},  $\Gamma^{\prime}[N_{\Gamma^{\prime}}(v_1)\backslash\{v_2,v_4,v_5\}]\cong ( K_{d_{\Gamma^{\prime}}(v_1)-3},+)$. We first claim that $2\leq |N_{\Gamma^{\prime}}(v_1)\cap N_{\Gamma^{\prime}}(v_5)|\leq 3$. Otherwise, $|N_{\Gamma^{\prime}}(v_1)\cap N_{\Gamma^{\prime}}(v_5)|\geq 4$, then $\Gamma^{\prime}$ contains an unbalanced $K_{3,3}$, a contradiction. Next, we assert that $v_4v_5,v_5v_6 \notin E(\Gamma^{\prime})$. Otherwise, $\Gamma^{\prime}$ contains an unbalanced $K_{3,3}$, a contradiction. Let $\Gamma^{\prime\prime}=\Gamma^{\prime}+v_4v_5+v_5v_6-v_1v_5-v_2v_5$, then $\Gamma^{\prime\prime}$ is a $\mathcal{K}_{3,3}^{-}$-free unbalanced signed graph. Note that $\lambda_1(x_4+x_6)\geq 2x_1+2x_3+x_2+x_4+x_6+\sum_{i\in N_{\Gamma^{\prime}}(v_1)\setminus\{v_2,v_3,v_4,v_5,v_6\}}x_i$, $\lambda_1(x_1+x_2)=-x_1-x_2+2x_3+2x_4+2x_5+x_6+\sum_{i\in N_{\Gamma^{\prime}}(v_1)\setminus\{v_2,v_3,v_4,v_5,v_6\}}x_i$. Then $\lambda_1(x_4+x_6-x_1-x_2)\geq 3x_1+2x_2-x_4-2x_5$. That is, $(\lambda_1+2)(x_4+x_6-x_1-x_2)\geq x_1+x_4+2x_6-2x_5$. Note that $\lambda_1(2x_6+x_4-2x_5)>2x_4+x_1+x_3+x_6-x_2>0$. Thus, $x_4+x_6-x_1-x_2>0$ and
\begin{align*}
	\lambda_1(A(\Gamma^{\prime\prime}))-\lambda_1(A(\Gamma^{\prime}))&\geq X^T(A(\Gamma^{\prime\prime})-A(\Gamma^{\prime}))X
	=2x_5(x_4+x_6-x_1-x_2)>0,
\end{align*} 
a contradiction. 

$(3)$ $N_{\Gamma^{\prime}}(v_1)\cap N_{\Gamma^{\prime}}(v_4)=\{v_2,v_3\}$, then we assert that $d_{\Gamma^{\prime}}(v_1)=5$. Otherwise, $d_{\Gamma^{\prime}}(v_1)\geq 6$ and $v_i\in N_{\Gamma^{\prime}}(v_1)$ for $2\leq i\leq 7$. Under this condition, we can claim that $|N_{\Gamma^{\prime}}(v_1)\cap N_{\Gamma^{\prime}}(v_5)|=3$. Otherwise, $|N_{\Gamma^{\prime}}(v_1)\cap N_{\Gamma^{\prime}}(v_5)|\neq 3$. If $|N_{\Gamma^{\prime}}(v_1)\cap N_{\Gamma^{\prime}}(v_5)|=2$, i.e., $N_{\Gamma^{\prime}}(v_1)\cap N_{\Gamma^{\prime}}(v_5)=\{v_2,v_3\}$, let $\Gamma^{\prime\prime}=\Gamma^{\prime}+v_5v_6$, then $\Gamma^{\prime\prime}$ is a $\mathcal{K}_{3,3}^{-}$-free unbalanced signed graph and $\lambda_1(A(\Gamma^{\prime\prime}))>\lambda_1(A(\Gamma^{\prime}))$ by $(i)$ of Lemma \ref{l5}, a contradiction. If $|N_{\Gamma^{\prime}}(v_1)\cap N_{\Gamma^{\prime}}(v_5)|\geq 4$, then $\Gamma^{\prime}$ contains an unbalanced $K_{3,3}$, a contradiction. Thus, $|N_{\Gamma^{\prime}}(v_1)\cap N_{\Gamma^{\prime}}(v_5)|=3$. Without loss of generality, assume that $N_{\Gamma^{\prime}}(v_1)\cap N_{\Gamma^{\prime}}(v_5)=\{v_2,v_3,v_6\}$. Let $\Gamma^{\prime\prime}=\Gamma^{\prime}+v_4v_7$, then $\Gamma^{\prime\prime}$ is a $\mathcal{K}_{3,3}^{-}$-free unbalanced signed graph and $\lambda_1(A(\Gamma^{\prime\prime}))>\lambda_1(A(\Gamma^{\prime}))$ by $(i)$ of Lemma \ref{l5}, a contradiction. Thus, $d_{\Gamma^{\prime}}(v_1)=5$. Assume that $N_{\Gamma^{\prime}}(v_1)=\{v_2,v_3,v_4,v_5,v_6\}$. By $(i)$ of Lemma \ref{l5}, $v_i$ is adjacent to every vertex in $V(\Gamma^{\prime})\backslash\{v_1,v_2\}$ for all $v_i\in V(\Gamma^{\prime})\backslash N_{\Gamma^{\prime}}[v_1]$ and $v_5v_6 \in E(\Gamma^{\prime})$.  Let $\Gamma^{\prime\prime}=\Gamma^{\prime}+v_4v_5+v_4v_6-v_1v_4-v_2v_4$, then $\Gamma^{\prime\prime}$ is a $\mathcal{K}_{3,3}^{-}$-free unbalanced signed graph. Note that $\lambda_1(x_5+x_6)>2x_1+2x_3+x_2+x_5+x_6$,  $\lambda_1(x_1+x_2)=-x_1-x_2+2x_3+2x_4+2x_5+x_6$. Then  $\lambda_1(x_5+x_6-x_1-x_2)>3x_1+2x_2-2x_4-x_5$. That is, $(\lambda_1+2)(x_4+x_6-x_1-x_2)>x_1+2x_6+x_5-2x_4$. Note that $\lambda_1x_4=\sum_{v_i\in V(\Gamma^{\prime})\setminus N_{\Gamma^{\prime}}[v_1]}x_i+x_1+x_2+x_3$, $\lambda_1(2x_6+x_5)>2(\sum_{v_i\in V(\Gamma^{\prime})\setminus N_{\Gamma^{\prime}}[v_1]}x_i)+3x_1+3x_3+x_2$. Then $\lambda_1(2x_6+x_5-2x_4)>x_1+x_3-x_2>0$. Thus, $x_5+x_6-x_1-x_2>0$ and
\begin{align*}
	\lambda_1(A(\Gamma^{\prime\prime}))-\lambda_1(A(\Gamma^{\prime}))&\geq X^T(A(\Gamma^{\prime\prime})-A(\Gamma^{\prime}))X
	=2x_4(x_5+x_6-x_1-x_2)>0,
\end{align*} 
a contradiction.
\begin{subcase}
	$d_{\Gamma^{\prime}}(v_1)\geq d_{\Gamma^{\prime}}(v_2)\geq 5$.
\end{subcase}
Firstly, we consider that $|N_{\Gamma^{\prime}}(v_1)\cap N_{\Gamma^{\prime}}(v_2)|=2, i.e., N_{\Gamma^{\prime}}(v_1)\cap N_{\Gamma^{\prime}}(v_2)=\{v_3,v_4\}$. By $(i)$ of Lemma \ref{l5}, $\Gamma^{\prime}[N_{\Gamma^{\prime}}(v_1)\backslash\{v_2,v_4\}]\cong  (K_{d_{\Gamma^{\prime}}(v_1)-2},+)$, $\Gamma^{\prime}[N_{\Gamma^{\prime}}(v_2)\backslash\{v_1,v_4\}]\cong (K_{d_{\Gamma^{\prime}}(v_2)-2},+)$, $v_i$ is adjacent to every vertex in $V(\Gamma^{\prime})\backslash\{v_1,v_2\}$ for all $v_i\in V(\Gamma^{\prime})\backslash(N_{\Gamma^{\prime}}[v_1]\cup N_{\Gamma^{\prime}}[v_2])$ and $|N_{\Gamma^{\prime}}(v_4)\cap N_{\Gamma^{\prime}}(v_1)|\leq3$, $|N_{\Gamma^{\prime}}(v_4)\cap N_{\Gamma^{\prime}}(v_2)|\leq3$ since $\Gamma^{\prime}$ is a $\mathcal{K}_{3,3}^{-}$-free unbalanced signed graph. Now, we will consider two subcases.

 $(1)$ $|N_{\Gamma^{\prime}}(v_4)\cap N_{\Gamma^{\prime}}(v_1)|=|N_{\Gamma^{\prime}}(v_4)\cap N_{\Gamma^{\prime}}(v_2)|=2$ or $|N_{\Gamma^{\prime}}(v_4)\cap N_{\Gamma^{\prime}}(v_1)|=2$, $|N_{\Gamma^{\prime}}(v_4)\cap N_{\Gamma^{\prime}}(v_2)|=3$ or $|N_{\Gamma^{\prime}}(v_4)\cap N_{\Gamma^{\prime}}(v_1)|=3$, $|N_{\Gamma^{\prime}}(v_4)\cap N_{\Gamma^{\prime}}(v_2)|=2$. If $x_1 \geq x_2$, let $\Gamma^{\prime\prime}=\Gamma^{\prime}+v_1w-v_2w$ for all $w\in N_{\Gamma^{\prime}}(v_2)\setminus\{v_1,v_3,v_4\}$, then $\Gamma^{\prime\prime}$ is a $\mathcal{K}_{3,3}^{-}$-free unbalanced signed graph and $\lambda_1(A(\Gamma^{\prime\prime}))>\lambda_1(A(\Gamma^{\prime}))$ by $(ii)$ of Lemma \ref{l5}, a contradiction. If $x_1 < x_2$, let $\Gamma^{\prime\prime}=\Gamma^{\prime}+v_2u-v_1u$ for all $u\in N_{\Gamma^{\prime}}(v_1)\setminus\{v_2,v_3,v_4\}$, then $\Gamma^{\prime\prime}$ is a $\mathcal{K}_{3,3}^{-}$-free unbalanced signed graph and $\lambda_1(A(\Gamma^{\prime\prime}))>\lambda_1(A(\Gamma^{\prime}))$ by $(ii)$ of Lemma \ref{l5}, a contradiction. 
 
 $(2)$ $|N_{\Gamma^{\prime}}(v_4)\cap N_{\Gamma^{\prime}}(v_1)|=|N_{\Gamma^{\prime}}(v_4)\cap N_{\Gamma^{\prime}}(v_2)|=3$. Without loss of generality, assume that $N_{\Gamma^{\prime}}(v_4)\cap N_{\Gamma^{\prime}}(v_1)=\{v_2,v_3,v_5\}$ and $N_{\Gamma^{\prime}}(v_4)\cap N_{\Gamma^{\prime}}(v_2)=\{v_1,v_3,v_7\}$. If $x_1\geq x_2$, let  $\Gamma^{\prime\prime}=\Gamma^{\prime}+v_1w-v_2w-v_4v_7+v_7u$ for all $w\in N_{\Gamma^{\prime}}(v_2)\setminus\{v_1,v_3,v_4\}$ and $u\in N_{\Gamma^{\prime}}(v_1)\setminus\{v_2,v_3,v_4\}$, then $\Gamma^{\prime\prime}$ is a $\mathcal{K}_{3,3}^{-}$-free unbalanced signed graph. Note that $\lambda_1(\sum_{u\in N_{\Gamma^{\prime}}(v_1)\setminus\{v_2,v_3,v_4\}}x_u)>2x_1+2x_3+x_5+\sum_{v_i\in V(\Gamma^{\prime})\setminus(N_{\Gamma^{\prime}}[v_1]\cup N_{\Gamma^{\prime}}[v_2])}x_i$, $\lambda_1x_4=x_1+x_2+x_3+x_5+x_7+\sum_{v_i\in V(\Gamma^{\prime})\setminus(N_{\Gamma^{\prime}}[v_1]\cup N_{\Gamma^{\prime}}[v_2])}x_i$. Thus, $\lambda_1(\sum_{u\in N_{\Gamma^{\prime}}(v_1)\setminus\{v_2,v_3,v_4\}}x_u-x_4)>x_1+x_3-x_2-x_7>0$ by $x_1\geq x_2$ and $x_3>x_7$. Then
\begin{align*}
	\lambda_1(A(\Gamma^{\prime\prime}))-\lambda_1(A(\Gamma^{\prime}))&\geq X^T(A(\Gamma^{\prime\prime})-A(\Gamma^{\prime}))X>2x_7(\sum_{u\in N_{\Gamma^{\prime}}(v_1)\setminus\{v_2,v_3,v_4\}}x_u-x_4)>0,
\end{align*} 
 a contradiction. If $x_1<x_2$, let  $\Gamma^{\prime\prime}=\Gamma^{\prime}+v_2u-v_1u-v_4v_5+v_5w$ for all $w\in N_{\Gamma^{\prime}}(v_2)\setminus\{v_1,v_3,v_4\}$ and $u\in N_{\Gamma^{\prime}}(v_1)\setminus\{v_2,v_3,v_4\}$, then $\Gamma^{\prime\prime}$ is a $\mathcal{K}_{3,3}^{-}$-free unbalanced signed graph. Note that $\lambda_1(\sum_{w\in N_{\Gamma^{\prime}}(v_2)\setminus\{v_1,v_3,v_4\}}x_w)>2x_2+2x_3+x_7+\sum_{v_i\in V(\Gamma^{\prime})\setminus(N_{\Gamma^{\prime}}[v_1]\cup N_{\Gamma^{\prime}}[v_2])}x_i$, $\lambda_1x_4=x_1+x_2+x_3+x_5+x_7+\sum_{v_i\in V(\Gamma^{\prime})\setminus(N_{\Gamma^{\prime}}[v_1]\cup N_{\Gamma^{\prime}}[v_2])}x_i$. Thus, $\lambda_1(\sum_{w\in N_{\Gamma^{\prime}}(v_2)\setminus\{v_1,v_3,v_4\}}x_w$ $-x_4)>x_2+x_3-x_1-x_5>0$ by $x_2>x_1$ and $x_3>x_5$. Then
\begin{align*}
	\lambda_1(A(\Gamma^{\prime\prime}))-\lambda_1(A(\Gamma^{\prime}))&\geq X^T(A(\Gamma^{\prime\prime})-A(\Gamma^{\prime}))X
	>2x_5(\sum_{w\in N_{\Gamma^{\prime}}(v_2)\setminus\{v_1,v_3,v_4\}}x_w-x_4)>0,
\end{align*} 
 a contradiction.
 
Secondly, we assume that $3\leq|N_{\Gamma^{\prime}}(v_1)\cap N_{\Gamma^{\prime}}(v_2)|<|N_{\Gamma^{\prime}}(v_2)|-1$ and set $M=N_{\Gamma^{\prime}}(v_1)\cap N_{\Gamma^{\prime}}(v_2)$. By $(i)$ of Lemma \ref{l5}, $\Gamma^{\prime}[N_{\Gamma^{\prime}}(v_1)\backslash(M\cup\{v_2\})]\cong(K_{d_{\Gamma^{\prime}}(v_1)-|M|-1},+)$, $\Gamma^{\prime}[N_{\Gamma^{\prime}}(v_2)\backslash(M\cup\{v_1\})]\cong(K_{d_{\Gamma^{\prime}}(v_2)-|M|-1},+)$, $v_i$ is adjacent to every vertex in $V(\Gamma^{\prime})\backslash\{v_1,v_2\}$ for all $v_i\in V(\Gamma^{\prime})\backslash$ $(N_{\Gamma^{\prime}}[v_1]\cup N_{\Gamma^{\prime}}[v_2])$ and $|N_{\Gamma^{\prime}}(v_4)\cap N_{\Gamma^{\prime}}(v_1)|$ $\leq3$, $|N_{\Gamma^{\prime}}(v_4)\cap N_{\Gamma^{\prime}}(v_2)|\leq3$ since $\Gamma^{\prime}$ is a $\mathcal{K}_{3,3}^{-}$-free unbalanced signed graph. Now, we will consider it in three subcases. 

$(1)$ $v_4$ is adjacent to a vertex in $M\backslash\{v_3,v_4\}$. If $x_1\geq x_2$, let  $\Gamma^{\prime\prime}=\Gamma^{\prime}+v_1w-v_2w$ for all $w\in N_{\Gamma^{\prime}}(v_2)\backslash(M\cup\{v_1\})$, then $\Gamma^{\prime\prime}$ is a $\mathcal{K}_{3,3}^{-}$-free unbalanced signed graph and $\lambda_1(A(\Gamma^{\prime\prime}))>\lambda_1(A(\Gamma^{\prime}))$ by $(ii)$ of Lemma \ref{l5}, a contradiction. If $x_1<x_2$, let  $\Gamma^{\prime\prime}=\Gamma^{\prime}+v_2u-v_1u$ for all $u\in N_{\Gamma^{\prime}}(v_1)\backslash(M\cup\{v_2\})$, then $\Gamma^{\prime\prime}$ is a $\mathcal{K}_{3,3}^{-}$-free unbalanced signed graph and $\lambda_1(A(\Gamma^{\prime\prime}))>\lambda_1(A(\Gamma^{\prime}))$ by $(ii)$ of Lemma \ref{l5}, a contradiction. 

$(2)$ $v_4$ is adjacent to a vertex in $(N_{\Gamma^{\prime}}(v_1)\cup N_{\Gamma^{\prime}}(v_2))\backslash(M\cup\{v_1,v_2\})$. Without loss of generality, assume that $N_{\Gamma^{\prime}}(v_2)\cap N_{\Gamma^{\prime}}(v_4)=\{v_1,v_3,v_7\}$ and $v_7\notin M$, then $v_7v_i \notin E(\Gamma^{\prime})$ for all $v_i\in N_{\Gamma^{\prime}}(v_1)\setminus\{v_2,v_3,v_4\}$ since $\Gamma^{\prime}$ is a $\mathcal{K}_{3,3}^{-}$-free unbalanced signed graph. By performing the same operation as in $(i)$, we can derive a contradiction. 

$(3)$ $v_4$ is adjacent to a vertex in $N_{\Gamma^{\prime}}(v_1)\backslash(M\cup\{v_2\})$ and a vertex in $N_{\Gamma^{\prime}}(v_2)\backslash(M\cup\{v_1\})$. Without loss of generality, assume that $N_{\Gamma^{\prime}}(v_1)\cap N_{\Gamma^{\prime}}(v_4)=\{v_2,v_3,v_6\}$, $N_{\Gamma^{\prime}}(v_2)\cap N_{\Gamma^{\prime}}(v_4)=\{v_1,v_3,v_7\}$ and $v_6,v_7\notin M$. Then $v_7v_i \notin E(\Gamma^{\prime})$ for all $v_i \in N_{\Gamma^{\prime}}(v_1)\backslash\{v_2,v_3,v_4\}$ and $v_6v_j \notin E(\Gamma^{\prime})$ for all $v_j \in N_{\Gamma^{\prime}}(v_2)\backslash\{v_1,v_3,v_4\}$ since $\Gamma^{\prime}$ is a $\mathcal{K}_{3,3}^{-}$-free unbalanced signed graph. Let  $\Gamma^{\prime\prime}=\Gamma^{\prime}+v_7v_i-v_2v_7$ for all $v_i \in N_{\Gamma^{\prime}}(v_1)\backslash\{v_2,v_3,v_4\}$, then $\Gamma^{\prime\prime}$ is a $\mathcal{K}_{3,3}^{-}$-free unbalanced signed graph. Note that $\lambda_1(\sum_{i\in N_{\Gamma^{\prime}}(v_1)\setminus\{v_2,v_3,v_4\}}x_i)>|N_{\Gamma^{\prime}}(v_1)\setminus\{v_2,v_3,v_4\}|x_3+\sum_{w\in N_{\Gamma^{\prime}}(v_1)\setminus(M\cup\{v_2\})}x_w+x_2+x_4$, $\lambda_1x_2=\sum_{a\in M\setminus\{v_3,v_4\}}x_a+\sum_{b\in N_{\Gamma^{\prime}}(v_2)\setminus(M\cup \{v_1,v_7\})}x_b-x_1+x_3+x_4+x_7$. It is obvious  that $|M\setminus\{v_3,v_4\}|+|N_{\Gamma^{\prime}}(v_1)\backslash(M\cup \{v_2\})|= |N_{\Gamma^{\prime}}(v_1)\backslash\{v_2,v_3,v_4\}|$ and $|M\setminus\{v_3,v_4\}|x_3\geq \sum_{a\in M\setminus\{v_3,v_4\}}x_a$. Since $d_{\Gamma^{\prime}}(v_1)\geq d_{\Gamma^{\prime}}(v_2)$, $|N_{\Gamma^{\prime}}(v_1)\backslash(M\cup\{v_2\})|\geq|N_{\Gamma^{\prime}}(v_2)\backslash(M\cup\{v_1\})|$. Thus, $|N_{\Gamma^{\prime}}(v_1)\backslash(M\cup \{v_2\})|x_3 \geq x_3+\sum_{b\in N_{\Gamma^{\prime}}(v_2)\setminus(M\cup \{v_1,v_7\})}x_b$ and $|N_{\Gamma^{\prime}}(v_1)\setminus\{v_2,v_3,v_4\}|x_3>\sum_{a\in M\setminus\{v_3,v_4\}}x_a+\sum_{b\in N_{\Gamma^{\prime}}(v_2)\setminus(M\cup \{v_1,v_7\})}x_b+x_3$. Clearly, $\sum_{w\in N_{\Gamma^{\prime}}(v_1)\setminus(M\cup\{v_2\})}$ $x_w+x_2>x_7$. Hence, $\sum_{i\in N_{\Gamma^{\prime}}(v_1)\setminus\{v_2,v_3,v_4\}}$ $x_i-x_2>0$ and
\begin{align*}
	\lambda_1(A(\Gamma^{\prime\prime}))-\lambda_1(A(\Gamma^{\prime}))&\geq X^T(A(\Gamma^{\prime\prime})-A(\Gamma^{\prime}))X
	=2(\sum_{i\in N_{\Gamma^{\prime}}(v_1)\setminus\{v_2,v_3,v_4\}}x_i-x_2)>0,
\end{align*} 
 a contradiction.

Finally, we consider $N_{\Gamma^{\prime}}[v_2]\subseteq N_{\Gamma^{\prime}}[v_1]$, then we will consider it in two subcases.

$(1)$ $N_{\Gamma^{\prime}}[v_2]\subsetneqq N_{\Gamma^{\prime}}[v_1]$. By $(i)$ of Lemma \ref{l5}, $\Gamma^{\prime}[N_{\Gamma^{\prime}}[v_1]\backslash N_{\Gamma^{\prime}}[v_2]]\cong(K_{d_{\Gamma^{\prime}}(v_1)-|N_{\Gamma^{\prime}}[v_2]|+1},+)$, $v_i$ is adjacent to every vertex in $V(\Gamma^{\prime})\backslash$ $\{v_1,v_2\}$ for all $v_i\in V(\Gamma^{\prime})\backslash N_{\Gamma^{\prime}}[v_1]$, $|N_{\Gamma^{\prime}}(v_4)\cap N_{\Gamma^{\prime}}(v_1)|\leq3$ and $d_{\Gamma^{\prime}[N_{\Gamma^{\prime}}[v_1]]}(v_i)\leq4$ for all $v_i\in N_{\Gamma^{\prime}}(v_2)\backslash\{v_1,v_3,v_4\}$ since $\Gamma^{\prime}$ is a $\mathcal{K}_{3,3}^{-}$-free unbalanced signed graph. Let $S=N_{\Gamma^{\prime}}(v_2)\backslash\{v_1,v_3,v_4\}$, $T=N_{\Gamma^{\prime}}[v_1]\backslash N_{\Gamma^{\prime}}[v_2]$. Clearly, $|S|\geq 2$ by $d_{\Gamma^{\prime}}(v_2)\geq 5$.  We first assert that there is at most one isolated vertex in subgraph $\Gamma^{\prime}[N_{\Gamma^{\prime}}[v_1]\backslash\{v_2,v_3\}]$. Otherwise, assume that $v_i, v_j$ are two isolated vertices in subgraph $\Gamma^{\prime}[N_{\Gamma^{\prime}}(v_1)\backslash\{v_2,v_3\}]$. Let $\Gamma^{\prime\prime}=\Gamma^{\prime}+v_iv_j$, then $\Gamma^{\prime\prime}$ is a $\mathcal{K}_{3,3}^{-}$-free unbalanced signed graph and $\lambda_1(A(\Gamma^{\prime\prime}))>\lambda_1(A(\Gamma^{\prime}))$ by $(i)$ of Lemma \ref{l5}, a contradiction. If $N_{\Gamma^{\prime}}(v_4)\cap S\neq \phi$, then we will further discuss in three subcases. $(a)$ $|S|\geq 2, |T|=2$. Without loss of generality, assume that $v_5,v_6\in S$, $v_7,v_8\in T$ and $v_4v_5\in E(\Gamma^{\prime})$. We first claim that  $|S|\geq 3$. Otherwise, $|S|=2$, by $(i)$ of Lemma \ref{l5}, assume that $v_6v_7\in E(\Gamma^{\prime})$. Let $\Gamma^{\prime\prime}=\Gamma^{\prime}+v_4v_8+v_5v_8+v_6v_8-v_1v_8$, then $\Gamma^{\prime\prime}$ is a $\mathcal{K}_{3,3}^{-}$-free unbalanced signed graph and $\lambda_1(A(\Gamma^{\prime\prime}))>\lambda_1(A(\Gamma^{\prime}))$, a contradiction. Thus, $|S|\geq 3$. Next, we consider that $N_{\Gamma^{\prime}}(v_7)\cap S\neq\phi$, assume that $v_6v_7\in  E(\Gamma^{\prime})$. Then $v_7w \notin E(\Gamma^{\prime})$ for $w \in S\backslash\{v_6\}$ since $\Gamma^{\prime}$ is a $\mathcal{K}_{3,3}^{-}$-free unbalanced signed graph.
Let $\Gamma^{\prime\prime}=\Gamma^{\prime}+v_7v_4+v_7w-v_7v_1$ for all $w\in S\setminus \{v_6\}$, then $\Gamma^{\prime\prime}$ is a $\mathcal{K}_{3,3}^{-}$-free unbalanced signed graph. Note that $\lambda_1x_1=-x_2+x_3+x_4+x_5+x_6+x_7+x_8+\sum_{u\in S\setminus\{v_5,v_6\}}x_u$, $\lambda_1x_4=x_1+x_2+x_3+x_5+\sum_{v_i\in V(\Gamma^{\prime})\setminus N_{\Gamma^{\prime}}[v_1]}$ $x_i$ and $\lambda_1(\sum_{w\in S\setminus\{v_6\}}x_w)>(|S|-1)x_1+(|S|-1)x_2+(|S|-1)x_3+x_4+A$, where $A$ is the sum of $(|S|-3)$ $x$-components in $x_8+\sum_{u\in S\setminus\{v_5,v_6\}}x_u$. Clearly, $\lambda_1(\sum_{w\in S\setminus\{v_6\}}x_w$ $+x_4)>|S|x_1+|S|x_2+|S|x_3+x_4+x_5+A$. Then $\lambda_1(\sum_{w\in S\setminus\{v_6\}}x_w+x_4-x_1)>|S|x_1+(|S|+1)x_2+(|S|-1)x_3-x_6-x_7+A-x_8-\sum_{u\in S\setminus\{v_5,v_6\}}x_u$. That is, $(\lambda_1+|S|)(\sum_{w\in S\setminus\{v_6\}}x_w+x_4-x_1)>|S|(\sum_{w\in S\setminus\{v_6\}}x_w+x_4)+(|S|+1)x_2+(|S|-1)x_3-x_6-x_7+A-x_8-\sum_{u\in S\setminus\{v_5,v_6\}}x_u$. It is evident that $(|S|-1)x_3>x_6+x_7$ and $|S|(\sum_{w\in S\setminus\{v_6\}}x_w+x_4)>x_8+\sum_{u\in S\setminus\{v_5,v_6\}}x_u-A$. Thus, $\sum_{w \in S\setminus \{v_6\}}x_w+x_4-x_1>0$ and
\begin{align*}
	\lambda_1(A(\Gamma^{\prime\prime}))-\lambda_1(A(\Gamma^{\prime}))&\geq X^T(A(\Gamma^{\prime\prime})-A(\Gamma^{\prime}))X
	=2x_7(\sum_{w \in S\setminus \{v_6\}}x_w+x_4-x_1)>0,
\end{align*} 
a contradiction. If $N_{\Gamma^{\prime}}(v_4)\cap S=\phi$ or $|T|=1$, then we can derive a contradiction through the same operation. $(b)$ $|S|=2, |T|\geq 3$. Without loss of generality, assume that $v_5,v_6\in S$ and $v_4v_5\in E(\Gamma^{\prime})$. Let $\Gamma^{\prime\prime}=\Gamma^{\prime}+v_5w-v_2v_5-v_4v_5$ for all $w\in T$, then $\Gamma^{\prime\prime}$ is a $\mathcal{K}_{3,3}^{-}$-free unbalanced signed graph. Note that $\lambda_1x_2=-x_1+x_3+x_4+x_5+x_6$, $\lambda_1x_4=x_1+x_2+x_3+x_5+\sum_{v_i\in V(\Gamma^{\prime})\setminus N_{\Gamma^{\prime}}[v_1]}x_i$ and $\lambda_1(\sum_{w\in T}x_w)>(|T|-1)(\sum_{w\in T}x_w)+\sum_{v_i\in V(\Gamma^{\prime})\backslash N_{\Gamma^{\prime}}[v_1]}x_i+3x_3+x_6$. Then $\lambda_1(\sum_{w\in T}x_w$ $-x_2-x_4)>x_3-x_2-x_4-2x_5+(|T|-1)(\sum_{w\in T}x_w)$. That is, $(\lambda_1-1)(\sum_{w\in T}x_w-x_2-x_4)>x_3+(|T|-2)(\sum_{w\in T}x_w)-2x_5$. Since $|T|-2\geq 1$, $(|T|-2)(\sum_{w\in T}x_w)+x_3>2x_5$. Thus, $\sum_{w\in T}x_w-x_2-x_4>0$ and
\begin{align*}
	\lambda_1(A(\Gamma^{\prime\prime}))-\lambda_1(A(\Gamma^{\prime}))&\geq X^T(A(\Gamma^{\prime\prime})-A(\Gamma^{\prime}))X
	=2x_5(\sum_{w\in T}x_w-x_2-x_4)>0,
\end{align*} 
 a contradiction. $(c)$ $|S|\geq 3, |T|\geq 3$. Without loss of generality, assume that $v_5\in S$, $v_c\in T$ and $v_4v_5\in E(\Gamma^{\prime})$. Let  $\Gamma^{\prime\prime}=\Gamma^{\prime}+v_cv_4+v_cw-v_cv_1$ for all $w\in S\setminus (S\cap N_{\Gamma^{\prime}}(v_c))$, then $\Gamma^{\prime\prime}$ is a $\mathcal{K}_{3,3}^{-}$-free unbalanced signed graph. Note that $\lambda_1(\sum_{w\in S\setminus(S\cap N_{\Gamma^{\prime}}(v_c))}x_w+x_4)>|S|x_3+x_4+A$, $\lambda_1x_1=-x_2+x_3+x_4+A+B$, where $A$ is the sum of $(|S|-2)$ $x$-components in $\sum_{v_j\in S\cup T}x_j$ and $A+B=\sum_{v_j\in S\cup T}x_j$. Then $\lambda_1(\sum_{w\in S\setminus(S\cap N_{\Gamma^{\prime}}(v_c))}x_w+x_4-x_1)>(|S|-1)x_3-B$. If $|S|-1\geq|T|+2$, then $(|S|-1)x_3-B>0$. Thus, $\sum_{w\in S\setminus(S\cap N_{\Gamma^{\prime}}(v_c))}x_w+x_4-x_1>0$ and
\begin{align*}
	\lambda_1(A(\Gamma^{\prime\prime}))-\lambda_1(A(\Gamma^{\prime}))&\geq X^T(A(\Gamma^{\prime\prime})-A(\Gamma^{\prime}))X
	=2x_c(\sum_{w\in S\setminus(S\cap N_{\Gamma^{\prime}}(v_c))}x_w+x_4-x_1)>0,
\end{align*} 
a contradiction. Thus, $|S|\leq|T|+2$. Let  $\Gamma^{\prime\prime}=\Gamma^{\prime}+v_5u-v_2v_5-v_4v_5$ for all $u\in T$, then $\Gamma^{\prime\prime}$ is a $\mathcal{K}_{3,3}^{-}$-free unbalanced signed graph. Note that $(\lambda_1-1)(x_2+x_4)=2x_3+2x_5+\sum_{k\in S\setminus\{v_5\}}x_k+\sum_{v_i\in V(\Gamma^{\prime})\setminus N_{\Gamma^{\prime}}[v_1]}x_i$, $(\lambda_1-1)(\sum_{u\in T}x_u)>|T|x_3+\sum_{u\in T}x_u+\sum_{v_i\in V(\Gamma^{\prime})\setminus N_{\Gamma^{\prime}}[v_1]}x_i$. Then $(\lambda_1-1)(\sum_{u\in T}x_u-x_2-x_4)>(|T|-2)x_3+\sum_{u\in T}x_u-\sum_{k\in S\setminus\{v_5\}}x_k-2x_5$. It is evident that $\sum_{u\in T}x_u>2x_5$. If $|T|-2\geq |S|-1$, then $\sum_{u\in T}x_u-x_2-x_4>0$. Hence,
\begin{align*}
	\lambda_1(A(\Gamma^{\prime\prime}))-\lambda_1(A(\Gamma^{\prime}))&\geq X^T(A(\Gamma^{\prime\prime})-A(\Gamma^{\prime}))X
	=2x_5(\sum_{u\in T}x_u-x_2-x_4)>0,
\end{align*} 
 a contradiction. Thus, $|T|\leq |S|\leq |T|+2$. If $|S|=|T|$, then there exists a vertex in $T$ that is not adjacent to any vertex in $S$ by $v_4v_5\in E(\Gamma^{\prime})$. Without loss of generality, assume that this vertex is $v_r\in T$. Let $\Gamma^{\prime\prime}=\Gamma^{\prime}+v_rv_4+v_rw-v_rv_1$ for all $w\in S$, then $\Gamma^{\prime\prime}$ is a $\mathcal{K}_{3,3}^{-}$-free unbalanced signed graph. Note that $\lambda_1(\sum_{w\in S}x_w+x_4)>(|S|+1)x_3+x_4+C$, $\lambda_1x_1=-x_2+x_3+x_4+C+D$, where $C$ is the sum of $|S|$ $x$-components in $\sum_{v_j\in S\cup T}x_j$ and $C+D=\sum_{v_j\in S\cup T}x_j$. Then $\lambda_1(\sum_{w\in S}x_w+x_4-x_1)>|S|x_3-D>0$ by $|S|=|T|$. Thus, $\sum_{w\in S}x_w+x_4-x_1>0$ and 
\begin{align*}
	\lambda_1(A(\Gamma^{\prime\prime}))-\lambda_1(A(\Gamma^{\prime}))&\geq X^T(A(\Gamma^{\prime\prime})-A(\Gamma^{\prime}))X
	=2x_r(\sum_{w\in S}x_w+x_4-x_1)>0,
\end{align*} 
 a contradiction. Hence, $|S|\neq|T|$. If $|S|=|T|+1$, without loss of generality, assume that $v_5,v_6\in S, v_7\in T$ and $v_4v_5\in E(\Gamma^{\prime})$. Through the discussion of the case where $|S|=|T|$, we have $\Gamma^{\prime}[S\cup T\setminus\{v_5\}]\cong K_t\circ K_1$. Otherwise, we can derive a contradiction through the same operation. Assume that $v_6v_7\in E(\Gamma^{\prime})$. Let  $\Gamma^{\prime\prime}=\Gamma^{\prime}+v_6w-v_2v_6$ for all $w\in T\backslash\{v_7\}$, then $\Gamma^{\prime\prime}$ is a $\mathcal{K}_{3,3}^{-}$-free unbalanced signed graph. Note that $\lambda_1(\sum_{w\in T\setminus\{v_7\}}x_w)>(|T|-1)x_3+\sum_{u\in  S\setminus\{v_5,v_6\}}x_u+\sum_{k\in T}x_k$, $\lambda_1x_2=-x_1+x_3+x_4+x_5+x_6+\sum_{u\in  S\setminus\{v_5,v_6\}}x_u$. Then $\lambda_1(\sum_{w\in T\setminus\{v_7\}}x_w-x_2)>(|T|-2)x_3+\sum_{k\in T}x_k-x_4-x_5-x_6$.
 Clearly, $(|T|-2)x_3> x_4$ by $|T|\geq 3$.  Note that $\lambda_1(x_5+x_6)=2x_1+2x_2+2x_3+x_4+x_7$, $\lambda_1(\sum_{k\in T}x_k)>|T|x_3+|T|x_1+x_7+\sum_{k\in T}x_k$. Since $|T|\geq 3$ and $x_1>x_2$, $|T|x_3+|T|x_1>2x_1+2x_2+2x_3$. It is evident that $\sum_{k\in T}x_k>x_4$. Thus, $\sum_{k\in T}x_k>x_5+x_6$ and $\sum_{w\in T\setminus\{v_7\}}x_w-x_2>0$ and
\begin{align*}
	\lambda_1(A(\Gamma^{\prime\prime}))-\lambda_1(A(\Gamma^{\prime}))&\geq X^T(A(\Gamma^{\prime\prime})-A(\Gamma^{\prime}))X
	=2x_6(\sum_{w\in T\setminus\{v_7\}}x_w-x_2)>0,
\end{align*} 
a contradiction. Hence, $|S|\neq |T|+1$. If $|S|=|T|+2$, without loss of generality, assume that $v_5,v_a\in S$ and $v_4v_5 \in E(\Gamma^{\prime})$. Through the discussion of the case where $|S|=|T|$, there exists an isolated vertex in subgraph $\Gamma^{\prime}[N_{\Gamma^{\prime}}(v_1)\backslash\{v_2,v_3\}]$, assume that this vertex is $v_a$. By $(i)$ of Lemma \ref{l5}, then $\Gamma^{\prime}[S\cup T\setminus\{v_5,v_a\}]\cong K_t\circ K_1$. Otherwise, we can derive a contradiction through the same operation. Let  $\Gamma^{\prime\prime}=\Gamma^{\prime}+v_aw-v_2v_a$ for all $w\in T$, then $\Gamma^{\prime\prime}$ is a $\mathcal{K}_{3,3}^{-}$-free unbalanced signed graph. Note that $\lambda_1(\sum_{w\in T}x_w)>|T|x_3+\sum_{u\in  S\setminus\{v_5,v_a\}}x_u+(|T|-1)\sum_{w\in T}x_w$, $\lambda_1x_2=-x_1+x_3+x_4+x_5+x_a+\sum_{u\in  S\setminus\{v_5,v_a\}}x_u$. Then $\lambda_1(\sum_{w\in T}x_w-x_2)>(|T|-1)x_3+(|T|-1)(\sum_{w\in T}x_w)-x_4-x_5-x_a$.
Clearly, $(|T|-1)\sum_{w\in T}x_w>x_a$ and $(|T|-1)x_3> x_4+x_5$ by $|T|\geq 3$.  Thus, $\sum_{w\in T}x_w-x_2>0$ and
\begin{align*}
	\lambda_1(A(\Gamma^{\prime\prime}))-\lambda_1(A(\Gamma^{\prime}))&\geq X^T(A(\Gamma^{\prime\prime})-A(\Gamma^{\prime}))X
	=2x_a(\sum_{w\in T}x_w-x_2)>0,
\end{align*} 
a contradiction. Thus, $|S|\neq |T|+2$ and $N_{\Gamma^{\prime}}(v_4)\cap S=\phi$.  Next, we consider that $N_{\Gamma^{\prime}}(v_4)\cap T\neq \phi$, then we will further discuss in four subcases. $(a)$ $|S|\geq 2, |T|=1$. Without loss of generality, assume that  $v_5\in S$, $v_6\in T$ and $v_4v_6\in E(\Gamma^{\prime})$. Clearly, $v_6w\notin E(\Gamma^{\prime})$ for all $w\in S$ since $\Gamma^{\prime}$ is a $\mathcal{K}_{3,3}^{-}$-free unbalanced signed graph. We first consider that there exists an isolated vertex in the subgraph $\Gamma^{\prime}[S]$, assume that this vertex is $v_5$. Let  $\Gamma^{\prime\prime}=\Gamma^{\prime}+v_6w-v_1v_6$ for all $w\in S$, then $\Gamma^{\prime\prime}$ is a $\mathcal{K}_{3,3}^{-}$-free unbalanced signed graph. Note that $\lambda_1x_1=-x_2+x_3+x_4+x_5+x_6+\sum_{w\in S\backslash\{v_5\}}x_w$,
$\lambda_1(\sum_{w\in S}x_w)\geq|S|x_1+|S|x_2+|S|x_3+\sum_{w\in S\setminus\{v_5\}}x_w$. Then $\lambda_1(\sum_{w\in S}x_w-x_1)\geq|S|x_1+(|S|+1)x_2+(|S|-1)x_3-x_4-x_5-x_6$. That is, $(\lambda_1+|S|)(\sum_{w\in S}x_w-x_1)\geq|S|(\sum_{w\in S}x_w)+(|S|+1)x_2+(|S|-1)x_3-x_4-x_5-x_6>0$. Thus, $\sum_{w\in S}x_w-x_1>0$ and
\begin{align*}
	\lambda_1(A(\Gamma^{\prime\prime}))-\lambda_1(A(\Gamma^{\prime}))&\geq X^T(A(\Gamma^{\prime\prime})-A(\Gamma^{\prime}))X
	=2x_6(\sum_{w\in S}x_w-x_1)>0,
\end{align*} 
 a contradiction. Next, we assume that there is no isolated vertex in subgraph $\Gamma^{\prime}[S]$, then we can derive a contradiction through the same operation. $(b)$ $|S|\geq 2, |T|=2$. Without loss of generality, assume that $v_5,v_6\in S$, $v_7,v_8\in T$ and $v_4v_7\in E(\Gamma^{\prime})$. We first consider that $N_{\Gamma^{\prime}}(v_8)\cap S\neq \phi$ and there exists an isolated vertex in the subgraph $\Gamma^{\prime}[S]$, without loss of generality, assume that $v_6v_8\in E(\Gamma^{\prime})$ and $v_5$ is an isolated vertex in the subgraph $\Gamma^{\prime}[S]$. Obviously, $v_8w\notin E(\Gamma^{\prime})$ for all $w\in(S\cup\{v_4\})\backslash\{v_6\}$ since $\Gamma^{\prime}$ is a $\mathcal{K}_{3,3}^{-}$-free unbalanced signed graph. Let  $\Gamma^{\prime\prime}=\Gamma^{\prime}+v_8w-v_1v_8$ for all $w\in(S\cup\{v_4\})\backslash\{v_6\}$, then $\Gamma^{\prime\prime}$ is a $\mathcal{K}_{3,3}^{-}$-free unbalanced signed graph. Note that $\lambda_1(\sum_{w\in(S\cup\{v_4\})\setminus\{v_6\}}x_w)\geq|S|x_1+|S|x_2+|S|x_3+x_7+\sum_{u\in S\setminus\{v_5,v_6\}}x_u$, $\lambda_1x_1=-x_2+x_3+x_4+x_5+x_6+\sum_{u\in S\setminus\{v_5,v_6\}}x_u+x_7+x_8$. Then $\lambda_1(\sum_{w\in(S\cup\{v_4\})\setminus\{v_6\}}x_w-x_1)\geq|S|x_1+(|S|+1)x_2+(|S|-1)x_3-x_4-x_5-x_6-x_8$. That is, $(\lambda_1+|S|)(\sum_{w\in(S\cup\{v_4\})\setminus\{v_6\}}x_w-x_1)\geq(|S|+1)x_2+(|S|-1)x_3+|S|(\sum_{w\in(S\cup\{v_4\})\setminus\{v_6\}}x_w)-x_4-x_5-x_6-x_8$. It is evident that $|S|(\sum_{w\in(S\cup\{v_4\})\setminus\{v_6\}}x_w)\geq 2(x_4+x_5)$, then $(\lambda_1+|S|)(\sum_{w\in(S\cup\{v_4\})\setminus\{v_6\}}x_w-x_1)\geq(|S|+1)x_2+(|S|-1)x_3+x_4+x_5-x_6-x_8$. Obviously, $x_4+x_5>x_8$ and $(|S|-1)x_3>x_6$ by $|S|\geq 2$. Thus, $\sum_{w\in(S\cup\{v_4\})\setminus\{v_6\}}x_w-x_1>0$ and
\begin{align*}
	\lambda_1(A(\Gamma^{\prime\prime}))-\lambda_1(A(\Gamma^{\prime}))&\geq X^T(A(\Gamma^{\prime\prime})-A(\Gamma^{\prime}))X
	=2x_8(\sum_{w\in(S\cup\{v_4\})\setminus\{v_6\}}x_w-x_1)>0,
\end{align*} 
 a contradiction. Next, we assume that $N_{\Gamma^{\prime}}(v_8)\cap S=\phi$ or there is no isolated vertex in the subgraph $\Gamma^{\prime}[S]$, then we can derive a contradiction through the same operation.
$(c)$ $|S|= 2, |T|\geq 3$. Without loss of generality, assume that $v_5,v_6\in S$, $v_7,v_8,v_9\in T$ and $v_4v_7\in E(\Gamma^{\prime})$. We first consider that $v_5v_6\in E(\Gamma^{\prime})$,  
let  $\Gamma^{\prime\prime}=\Gamma^{\prime}+v_5w+v_6w-v_2v_5-v_2v_6$ for all $w\in T$, then $\Gamma^{\prime\prime}$ is a $\mathcal{K}_{3,3}^{-}$-free unbalanced signed graph. Note that $\lambda_1x_2=-x_1+x_3+x_4+x_5+x_6$, $\lambda_1(\sum_{w\in T}x_w)>3x_1+3x_3+x_4$. Then $\lambda_1(\sum_{w\in T}x_w-x_2)\geq4x_1+2x_3-x_5-x_6>0$. Thus, $\sum_{w\in t}x_w-x_2>0$ and
\begin{align*}
	\lambda_1(A(\Gamma^{\prime\prime}))-\lambda_1(A(\Gamma^{\prime}))&\geq X^T(A(\Gamma^{\prime\prime})-A(\Gamma^{\prime}))X
	=2(x_5+x_6)(\sum_{w\in T}x_w-x_2)>0,
\end{align*} 
 a contradiction. Thus, $v_5v_6\notin E(\Gamma^{\prime})$. By $(i)$ of Lemma \ref{l5}, assume that $v_5v_8,v_6v_9\in E(\Gamma^{\prime})$. Let  $\Gamma^{\prime\prime}=\Gamma^{\prime}+v_5w+v_6u-v_2v_5-v_2v_6$ for all $w\in T\setminus\{v_8\}$, $u\in T\setminus\{v_9\}$, then $\Gamma^{\prime\prime}$ is a $\mathcal{K}_{3,3}^{-}$-free unbalanced signed graph. Note that $\lambda_1x_2=-x_1+x_3+x_4+x_5+x_6$, $\lambda_1(\sum_{w\in T\setminus\{v_8\}}x_w)>2x_1+2x_3+x_4+x_6$ and $\lambda_1(\sum_{u\in T\setminus\{v_9\}}x_u)>2x_1+2x_3+x_4+x_5$. Then $\lambda_1(\sum_{w\in T\setminus\{v_8\}}x_w-x_2)>3x_1+x_3-x_5>0$, $\lambda_1(\sum_{u\in T\setminus\{v_9\}}x_u-x_2)>3x_1+x_3-x_6>0$ and
\begin{align*}
	\lambda_1(A(\Gamma^{\prime\prime}))-\lambda_1(A(\Gamma^{\prime}))&\geq X^T(A(\Gamma^{\prime\prime})-A(\Gamma^{\prime}))X\\
	&=2x_5(\sum_{w\in T\setminus\{v_8\}}x_w-x_2)+2x_6(\sum_{u\in T\setminus\{v_9\}}x_u-x_2)\\&>0,
\end{align*} 
 a contradiction. $(d)$ $|S|\geq 3, |T|\geq 3$. Without loss of generality, assume that $v_5,v_6, v_7\in S$, $v_8,v_9,v_{10}\in T$ and $v_4v_8\in E(\Gamma^{\prime})$. We first consider that $N_{\Gamma^{\prime}}(v_{10})\cap S\neq \phi$, assume that $v_{10}v_{5}\in E(\Gamma^{\prime})$. Let  $\Gamma^{\prime\prime}=\Gamma^{\prime}+v_{10}w-v_1v_{10}$ for all $w\in(S\cup\{v_4\})\backslash\{v_5\}$, then $\Gamma^{\prime\prime}$ is a $\mathcal{K}_{3,3}^{-}$-free unbalanced signed graph. Note that $\lambda_1(\sum_{w\in(S\cup\{v_4\})\setminus\{v_5\}}x_w)>|S|x_3+A$, $\lambda_1x_1=-x_2+x_3+x_4+A+B$, where $A$ is the sum of $(|S|-1)$ $x$-components in $\sum_{v_j\in S\cup T}x_j$ and $A+B=\sum_{v_j\in S\cup T}x_j$. Then $\lambda_1(\sum_{w\in(S\cup\{v_4\})\setminus\{v_5\}}x_w-x_1)>(|S|-1)x_3-x_4-B$. If $|S|-1\geq |T|+2$, then $\sum_{w\in(S\cup\{v_4\})\backslash\{v_5\}}$ $x_w-x_1>0$ and
\begin{align*}
	\lambda_1(A(\Gamma^{\prime\prime}))-\lambda_1(A(\Gamma^{\prime}))&\geq X^T(A(\Gamma^{\prime\prime})-A(\Gamma^{\prime}))X
	=2x_{10}(\sum_{w\in(S\cup\{v_4\})\backslash\{v_5\}}x_w-x_1)>0,
\end{align*} 
 a contradiction. So, $|S|\leq |T|+2$. Next, we assume that $N_{\Gamma^{\prime}}(v_{10})\cap S=\phi$, let $\Gamma^{\prime\prime}=\Gamma^{\prime}+v_{10}w-v_1v_{10}$ for all $w\in(S\cup\{v_4\})$, then $\Gamma^{\prime\prime}$ is a $\mathcal{K}_{3,3}^{-}$-free unbalanced signed graph. Similarly, if $|S|\geq |T|+1$, we can derive a contradiction. So, $|S|\leq |T|$. Through the above discussion, we have $|S|\leq |T|+2$. Now, we assert that $N_{\Gamma^{\prime}}(v_{5})\cap (S\backslash\{v_5\})=\phi$. Otherwise, $N_{\Gamma^{\prime}}(v_{5})\cap (S\backslash\{v_5\})\neq \phi$, assume that $v_5v_6\in E(\Gamma^{\prime})$. Let  $\Gamma^{\prime\prime}=\Gamma^{\prime}+v_5w-v_2v_5$ for all $w\in T$, then $\Gamma^{\prime\prime}$ is a $\mathcal{K}_{3,3}^{-}$-free unbalanced signed graph. Note that $\lambda_1(\sum_{w\in T}x_w)\geq|T|x_3+x_4+(|T|-1)(\sum_{w\in T}x_w)$, $\lambda_1x_2=-x_1+x_3+x_4+x_5+x_6+x_7+\sum_{u\in S\backslash\{v_5,v_6,v_7\}}x_u$. Then $\lambda_1(\sum_{w\in T}x_w-x_2)>(|T|-1)x_3+(|T|-1)(\sum_{w\in T}x_w)-x_5-x_6-x_7-\sum_{u\in S\backslash\{v_5,v_6,v_7\}}x_u$. It is evident that $(|T|-1)(\sum_{w\in T}x_w)>x_5+x_6+x_7$. Since $|T|\geq |S|-2$, $(|T|-1)x_3>\sum_{u\in S\backslash\{v_5,v_6,v_7\}}x_u$. Thus, $\sum_{w\in T}x_w-x_2>0$ and
\begin{align*}
	\lambda_1(A(\Gamma^{\prime\prime}))-\lambda_1(A(\Gamma^{\prime}))&\geq X^T(A(\Gamma^{\prime\prime})-A(\Gamma^{\prime}))X
	=2x_5(\sum_{w\in T}x_w-x_2)>0,
\end{align*} 
a contradiction. Thus, $N_{\Gamma^{\prime}}(v_{5})\cap (S\backslash\{v_5\})=\phi$. This implies that $N_{\Gamma^{\prime}}(v_{5})\cap T\neq \phi$ by $(i)$ of Lemma \ref{l5}. Assume that $v_5v_9\in E(\Gamma^{\prime})$. Let $\Gamma^{\prime\prime}=\Gamma^{\prime}+v_5w-v_2v_5$ for all $w\in T\backslash\{v_9\}$, then $\Gamma^{\prime\prime}$ is a $\mathcal{K}_{3,3}^{-}$-free unbalanced signed graph. Similarly, if $|T|\geq |S|-1$, we can derive a contradiction. So, $|T|\leq |S|-2$. Clearly, $|T|\geq |S|-2$, then $|T|=|S|-2$.  This implies that there is a vertex $v_a\in S$ such that $N_{\Gamma^{\prime}}(v_{a})\cap T= \phi$. Let $\Gamma^{\prime\prime}=\Gamma^{\prime}+v_aw-v_2v_a$ for all $w\in T$, then $\Gamma^{\prime\prime}$ is a $\mathcal{K}_{3,3}^{-}$-free unbalanced signed graph.  According to the above discussion, we get a contradiction. Thus, $N_{\Gamma^{\prime}}(v_4)\cap T=\phi$. Finally, we consider that $N_{\Gamma^{\prime}}(v_{4})\cap (S\cup T)=\phi$. Let $\Gamma^{\prime\prime}=\Gamma^{\prime}+v_{4}w-v_2v_{4}-v_1v_4$ for all $w\in S\cup T$, then $\Gamma^{\prime\prime}$ is a $\mathcal{K}_{3,3}^{-}$-free unbalanced signed graph. Similarly, we have $\lambda_1(A(\Gamma^{\prime\prime}))>\lambda_1(A(\Gamma^{\prime}))$, a contradiction.

 $(2)$ $N_{\Gamma^{\prime}}[v_2]= N_{\Gamma^{\prime}}[v_1]$. Let $S=N_{\Gamma^{\prime}}(v_1)\backslash\{v_2,v_3,v_4\}$, then $|S|\geq2$ by $d_{\Gamma^{\prime}}(v_1)\geq5$. Note that $d_{[\Gamma^{\prime}[S]\cup \{v_4\}]}(v_i)\leq 1$ for all $v_i\in S\cup \{v_4\}$ since $\Gamma^{\prime}$ is a $\mathcal{K}_{3,3}^{-}$-free unbalanced signed graph. By $(i)$ of Lemma \ref{l5}, $v_i$ is adjacent to every vertex in $V(\Gamma^{\prime})\backslash\{v_1,v_2\}$ for all $v_i\in V(\Gamma^{\prime})\backslash N_{\Gamma^{\prime}}[v_1]$ and there is at most one isolated vertex in the subgraph $\Gamma^{\prime}[S\cup \{v_4\}]$. Otherwise,  assume that $u,v$ are two isolated vertices in the subgraph $\Gamma^{\prime}[S\cup \{v_4\}]$. Let $\Gamma^{\prime\prime}=\Gamma^{\prime}+uv$, then $\Gamma^{\prime\prime}$ is a $\mathcal{K}_{3,3}^{-}$-free unbalanced signed graph and $\lambda_1(A(\Gamma^{\prime\prime}))>\lambda_1(A(\Gamma^{\prime}))$ by $(i)$ of Lemma \ref{l5}, a contradiction.  Next, we will further discuss in two subcases.
$(a)$ $N_{\Gamma^{\prime}}(v_4)\cap S=\phi$. This implies that there is no isolated vertex in the subgraph $\Gamma^{\prime}_{[S]}$. So, $|S|$ is even and subgraph $\Gamma^{\prime}_{[S]}\cong \frac{|S|}{2}P_2$. If $|S|=2$, without loss of generality, assume that $v_5,v_6\in S$ and $v_5v_6 \in E(\Gamma^{\prime})$. Let $\Gamma^{\prime\prime}=\Gamma^{\prime}+v_4v_5+v_4v_6-v_1v_4-v_2v_4$, then $\Gamma^{\prime\prime}$ is a $\mathcal{K}_{3,3}^{-}$-free unbalanced signed graph. Note that $\lambda_1(x_1+x_2)=-x_1-x_2+2x_3+2x_4+2x_5+2x_6$, $\lambda_1(x_5+x_6)=2x_1+2x_2+2x_3+x_5+x_6+2(\sum_{v_i\in V(\Gamma^{\prime})\setminus N_{\Gamma^{\prime}}[v_1]}x_i)$. Then $\lambda_1(x_5+x_6-x_1-x_2)>3x_1+3x_2-2x_4-x_5-x_6$. That is, $(\lambda_1+3)(x_5+x_6-x_1-x_2)>2(x_5+x_6-x_4)$. Since $\lambda_1x_4=x_1+x_2+x_3+\sum_{v_i\in V(\Gamma^{\prime})\setminus N_{\Gamma^{\prime}}[v_1]}x_i$, $\lambda_1(x_5+x_6-x_4)>x_1+x_2+x_3+x_5+x_6>0$. Thus, $x_5+x_6-x_1-x_2>0$ and
\begin{align*}
	\lambda_1(A(\Gamma^{\prime\prime}))-\lambda_1(A(\Gamma^{\prime}))&\geq X^T(A(\Gamma^{\prime\prime})-A(\Gamma^{\prime}))X
	=2x_4(x_5+x_6-x_1-x_2)>0,
\end{align*} 
 a contradiction. So, $|S|\neq 2$. If $|S|\geq 4$, let $\Gamma^{\prime\prime}=\Gamma^{\prime}+v_4w-v_1v_4-v_2v_4$ for all $w\in S$, then $\Gamma^{\prime\prime}$ is a $\mathcal{K}_{3,3}^{-}$-free unbalanced signed graph. Note that $\lambda_1(x_1+x_2)=-x_1-x_2+2x_3+2x_4+2(\sum_{w\in S}x_w)$, $\lambda_1(\sum_{w\in S}x_w)\geq |S|x_1+|S|x_2+|S|x_3+\sum_{w\in S}x_w$. Then $\lambda_1(\sum_{w\in S}x_w-x_1-x_2)\geq (|S|+1)x_1+(|S|+1)x_2+(|S|-2)x_3-2x_4-\sum_{w\in S}x_w$. That is, $(\lambda_1+|S|+1)(\sum_{w\in S}x_w-x_1-x_2)\geq(|S|-2)x_3-2x_4+|S|\sum_{w\in S}x_w$. Since $|S|\geq 4$, $(|S|-2)x_3-2x_4>0$. Thus, $\sum_{w\in S}x_w-x_1-x_2>0$ and
\begin{align*}
	\lambda_1(A(\Gamma^{\prime\prime}))-\lambda_1(A(\Gamma^{\prime}))&\geq X^T(A(\Gamma^{\prime\prime})-A(\Gamma^{\prime}))X
	=2x_4(\sum_{w\in S}x_w-x_1-x_2)>0,
\end{align*} 
a contradiction.  $(b)$ $N_{\Gamma^{\prime}}(v_4)\cap S\neq\phi$. Without loss of generality, assume that $v_5,v_6\in S$ and $v_4v_5 \in E(\Gamma^{\prime})$. We first assert that there is no isolated vertex in the subgraph $\Gamma^{\prime}[S\cup\{v_4\}]$. Otherwise, assume that $v_6$ is an isolated vertex in the subgraph $\Gamma^{\prime}[S\cup\{v_4\}]$. Then $|S|$ is even and subgraph $\Gamma^{\prime}_{[(S\cup \{v_4\})\setminus \{v_6\}]}\cong \frac{|S|}{2}P_2$. If $|S|=2$, let  $\Gamma^{\prime\prime}=\Gamma^{\prime}+v_4v_6+v_5v_6-v_1v_6-v_2v_6$, then $\Gamma^{\prime\prime}$ is a $\mathcal{K}_{3,3}^{-}$-free unbalanced signed graph. Note that $\lambda_1(x_1+x_2)=-x_1-x_2+2x_3+2x_4+2x_5+2x_6$, $\lambda_1(x_4+x_5)=2x_1+2x_2+2x_3+x_4+x_5+2(\sum_{v_i\in V(\Gamma^{\prime})\setminus N_{\Gamma^{\prime}}[v_1]}x_i)$. Then $\lambda_1(x_4+x_5-x_1-x_2)=3x_1+3x_2+2(\sum_{v_i\in V(\Gamma^{\prime})\setminus N_{\Gamma^{\prime}}[v_1]}x_i)-x_4-x_5-2x_6$. That is, $(\lambda_1+1)(x_4+x_5-x_1-x_2)=2(x_1+x_2+\sum_{v_i\in V(\Gamma^{\prime})\setminus N_{\Gamma^{\prime}}[v_1]}x_i-x_6)$. Note that $\lambda_1x_6=x_1+x_2+x_3+\sum_{v_i\in V(\Gamma^{\prime})\setminus N_{\Gamma^{\prime}}[v_1]}x_i$, $\lambda_1(\sum_{v_i\in V(\Gamma^{\prime})\setminus N_{\Gamma^{\prime}}[v_1]}x_i+x_1+x_2)>-x_1-x_2+2x_3+2x_4+2x_5+2x_6+\sum_{v_i\in V(\Gamma^{\prime})\setminus N_{\Gamma^{\prime}}[v_1]}x_i$. Then $\lambda_1(x_1+x_2+\sum_{v_i\in V(\Gamma^{\prime})\setminus N_{\Gamma^{\prime}}[v_1]}x_i-x_6)>-2x_1-2x_2+x_3+2x_4+2x_5+2x_6$. That is, $(\lambda_1+2)(x_1+x_2+\sum_{v_i\in V(\Gamma^{\prime})\setminus N_{\Gamma^{\prime}}[v_1]}x_i-x_6)>x_3+2x_4+2x_5+2(\sum_{v_i\in V(\Gamma^{\prime})\setminus N_{\Gamma^{\prime}}[v_1]}x_i)>0$. Thus, $x_4+x_5-x_1-x_2>0$ and
\begin{align*}
	\lambda_1(A(\Gamma^{\prime\prime}))-\lambda_1(A(\Gamma^{\prime}))&\geq X^T(A(\Gamma^{\prime\prime})-A(\Gamma^{\prime}))X
	=2x_6(x_4+x_5-x_1-x_2)>0,
\end{align*} 
 a contradiction. So, $|S|\neq 2$.  If $|S|\geq 4$, let  $\Gamma^{\prime\prime}=\Gamma^{\prime}+v_6v_4+v_6w-v_1v_6-v_2v_6$ for all $w\in S\setminus\{v_6\}$, then $\Gamma^{\prime\prime}$ is a $\mathcal{K}_{3,3}^{-}$-free unbalanced signed graph. Note that $\lambda_1(x_1+x_2)=-x_1-x_2+2x_3+2x_4+2x_6+2(\sum_{w\in S\setminus\{v_6\}}x_w)$, $\lambda_1(\sum_{w\in S\setminus\{v_6\}}x_w+x_4)\geq|S|x_1+|S|x_2+|S|x_3+\sum_{w\in S\setminus\{v_6\}}x_w+x_4$. Then $\lambda_1(\sum_{w\in S\setminus\{v_6\}}x_w+x_4-x_1-x_2)\geq(|S|+1)x_1+(|S|+1)x_2+(|S|-2)x_3-\sum_{w\in S\setminus\{v_6\}}x_w-x_4-2x_6$. That is, $(\lambda_1+1)(\sum_{w\in S\setminus\{v_6\}}x_w+x_4-x_1-x_2)\geq|S|x_1+|S|x_2+(|S|-2)x_3-2x_6$. Since$|S|\geq 4$, $(|S|-2)x_3-2x_6>0$. It is evident that $\sum_{w\in S\setminus\{v_6\}}x_w+x_4-x_1-x_2>0$ and
\begin{align*}
	\lambda_1(A(\Gamma^{\prime\prime}))-\lambda_1(A(\Gamma^{\prime}))&\geq X^T(A(\Gamma^{\prime\prime})-A(\Gamma^{\prime}))X
	=2(\sum_{w\in S\setminus\{v_6\}}x_w+x_4-x_1-x_2)>0,
\end{align*}  
 a contradiction. Hence, there is no isolated vertex in the subgraph $\Gamma^{\prime}[S\cup\{v_4\}]$. This implies that $|S|$ is odd and  subgraph $\Gamma^{\prime}_{[S\cup \{v_4\}]}\cong \frac{|S|+1}{2}P_2$. Next, we assert that $|S|=3$. Otherwise,  $|S|\geq 5$, let $\Gamma^{\prime\prime}=\Gamma^{\prime}+v_5u-v_1v_5-v_2v_5$ for all $u\in S\setminus\{v_5\}$, then $\Gamma^{\prime\prime}$ is a $\mathcal{K}_{3,3}^{-}$-free unbalanced signed graph. Note that $\lambda_1(\sum_{u\in S\setminus\{v_5\}}x_u)\geq(|S|-1)x_1+(|S|-1)x_2+(|S|-1)x_3+\sum_{u\in S\setminus\{v_5\}}x_u$, $\lambda_1(x_1+x_2)=-x_1-x_2+2x_3+2x_4+2x_5+2(\sum_{u\in S\setminus\{v_5\}}x_u)$. Then $\lambda_1(\sum_{u\in S\setminus\{v_5\}}x_u-x_1-x_2)\geq|S|x_1+|S|x_2+(|S|-3)x_3-2x_4-2x_5-\sum_{u\in S\setminus\{v_5\}}x_u$. That is, $(\lambda_1+|S|)(\sum_{u\in S\setminus\{v_5\}}x_u-x_1-x_2)\geq(|S|-3)x_3-2x_4-2x_5+(|S|-1)(\sum_{u\in S\setminus\{v_5\}}x_u)$. Since $|S|\geq 5$, $(|S|-3)x_3>2x_5$. It is evident that $(|S|-1)\sum_{u\in S\setminus\{v_5\}}x_u>2x_4$. Thus, $\sum_{u\in S\setminus\{v_5\}}x_u-x_1-x_2>0$ and
\begin{align*}
	\lambda_1(A(\Gamma^{\prime\prime}))-\lambda_1(A(\Gamma^{\prime}))&\geq X^T(A(\Gamma^{\prime\prime})-A(\Gamma^{\prime}))X
	=2(\sum_{u\in S\setminus\{v_5\}}x_u-x_1-x_2)>0,
\end{align*}  
 a contradiction. So, $|S|=3$. Without loss of generality, assume that $v_5,v_6,v_7\in S$ and $v_4v_5, v_6v_7 \in E(\Gamma^{\prime})$ by $(i)$ of Lemma \ref{l5}. Let $\Gamma^{\prime\prime}=\Gamma^{\prime}+v_5v_6+v_5v_7-v_1v_5-v_2v_5$, then $\Gamma^{\prime\prime}$ is a $\mathcal{K}_{3,3}^{-}$-free unbalanced signed graph. Note that $\lambda_1(x_6+x_7)=2x_1+2x_2+2x_3+x_6+x_7+2(\sum_{v_i\in V(\Gamma^{\prime})\setminus N_{\Gamma^{\prime}}[v_1]}x_i)$, $\lambda_1(x_1+x_2)=-x_1-x_2+2x_3+2x_4+2x_5+2x_6+2x_7$. Then $\lambda_1(x_6+x_7-x_1-x_2)\geq3x_1+3x_2-2x_4-2x_5-x_6-x_7$. That is, $(\lambda_1+3)(x_6+x_7-x_1-x_2)\geq2(x_6+x_7-x_4-x_5)$. Since $\lambda_1(x_4+x_5)=2x_1+2x_2+2x_3+x_4+x_5+2(\sum_{v_i\in V(\Gamma^{\prime})\setminus N_{\Gamma^{\prime}}[v_1]}x_i)$, $\lambda_1(x_6+x_7-x_4-x_5)=x_6+x_7-x_4-x_5$. That is, $(\lambda_1-1)(x_6+x_7-x_4-x_5)=0$. So, $x_6+x_7-x_4-x_5=0$ by $\lambda_1(A(\Gamma^\prime))\geq n-2$. Thus, $x_6+x_7-x_1-x_2\geq 0$ and
\begin{align*}
	\lambda_1(A(\Gamma^{\prime\prime}))-\lambda_1(A(\Gamma^{\prime}))&\geq X^T(A(\Gamma^{\prime\prime})-A(\Gamma^{\prime}))X
	=2x_5(x_6+x_7-x_1-x_2)\geq 0.
\end{align*} 
 If $\lambda_1(A(\Gamma^{\prime\prime}))= \lambda_1(A(\Gamma^{\prime}))$, let $\Gamma^{\prime\prime\prime}=\Gamma^{\prime\prime}+v_4v_6+v_4v_7-v_1v_4-v_2v_4$, then $\Gamma^{\prime\prime\prime}$ is a $\mathcal{K}_{3,3}^{-}$-free unbalanced signed graph. Similarly, we have $\lambda_1(A(\Gamma^{\prime\prime\prime}))>\lambda_1(A(\Gamma^{\prime\prime}))=\lambda_1(A(\Gamma^{\prime}))$, a contradiction. This completes the proof.\\\\
 {\bf{Author Contributions}} All authors jointly worked on the results and they read and approved the final manuscript.\\\\
 {\bf{Data Availability}} Not applicable. The manuscript has no associated data.
 \section*{Acknowledgements}
 The authors would like to show great gratitude to anonymous referees for their valuable suggestions which lead to an improvement of the original manuscript.
 \section*{Declarations}
 
 \noindent{\bf{Conflict of interest}} The authors declare that they have no conflict of interest.

\end{document}